\documentclass[a4paper,reqno, twopage]{amsart}

\usepackage[british]{babel}
\usepackage{csquotes} 
\numberwithin{equation}{section}
\usepackage{amssymb}
\usepackage[protrusion=true,expansion=true]{microtype}	
\usepackage{amsmath, amsfonts, amsthm} 
\usepackage[pdftex]{graphicx}	
\usepackage{url}
\usepackage[title,titletoc,page]{appendix} 
\usepackage{listings} 
\usepackage{color}
\usepackage{comment}
\usepackage{caption}
\usepackage{subcaption}
\usepackage{algorithm}

\usepackage{fancyhdr}
\usepackage[left=3.5cm,right=3.5cm,top=3.5cm,bottom=3.5cm,footskip=1.5cm,headsep=1cm]{geometry}
\usepackage[%
pdfauthor={Yannick De Bruijn, Erik Orvehed Hiltunen},%
pdftitle={PHOTONIC BANDGAPS},
pdfsubject={PHOTONIC BANDGAPS},
pdfkeywords={Subwavelength resonances, evanescent modes, band gap, interface eigenmodes, Bloch theory, complex Brillouin zone, layer potentials}]{hyperref}

\usepackage[T1]{fontenc} 
\usepackage{lmodern} 
\rmfamily 
\DeclareFontShape{T1}{lmr}{b}{sc}{<->ssub*cmr/bx/sc}{}
\DeclareFontShape{T1}{lmr}{bx}{sc}{<->ssub*cmr/bx/sc}{}
\usepackage{lipsum}
\usepackage{mathtools}
\usepackage{amsmath}
\usepackage{amssymb}
\usepackage{tikz}
\usetikzlibrary{matrix,positioning,decorations.pathreplacing,calc}
\usetikzlibrary{
	decorations.text,%
	decorations.markings,%
	shadows}
\usepackage{adjustbox}
\usepackage{graphicx}

\usepackage{dsfont} 
\usepackage{caption}
\usepackage{subcaption}
\usepackage[normalem]{ulem}

\usepackage{bm}

\graphicspath{{figures/}}

\usepackage{xargs}                      
\usepackage[prependcaption,textsize=tiny,textwidth=3cm]{todonotes}
\newcommandx{\unsure}[2][1=]{\todo[linecolor=red,backgroundcolor=red!25,bordercolor=red,#1]{#2}}
\newcommandx{\change}[2][1=]{\todo[linecolor=blue,backgroundcolor=blue!25,bordercolor=blue,#1]{#2}}
\newcommandx{\info}[2][1=]{\todo[linecolor=OliveGreen,backgroundcolor=OliveGreen!25,bordercolor=OliveGreen,#1]{#2}}
\newcommandx{\improvement}[2][1=]{\todo[linecolor=black,backgroundcolor=black!25,bordercolor=black,#1]{#2}}
\newcommandx{\thiswillnotshow}[2][1=]{\todo[disable,#1]{#2}}
\usepackage{amsthm}
\usepackage[noabbrev,capitalize]{cleveref}
\crefname{proposition}{Proposition}{Propositions}
\crefname{equation}{}{}

\newtheorem{theorem}{Theorem}[section]
\newtheorem{lemma}[theorem]{Lemma}
\newtheorem{proposition}[theorem]{Proposition}
\newtheorem{corollary}[theorem]{Corollary}

\theoremstyle{definition}
\newtheorem{definition}[theorem]{Definition}

\newtheorem{remark}[theorem]{Remark}

\crefname{assumption}{Assumption}{Assumptions}
\crefname{definition}{Definition}{Definitions}
\crefname{corollary}{Corollary}{Corollaries}
\crefname{enumi}{item}{items}

\usepackage{fancyhdr}

\fancypagestyle{plain}{%
	\fancyhead[C]{}
	\fancyfoot[C]{}
	
	}
\fancypagestyle{nosection}{%
	\fancyhead[CE]{}
	\fancyhead[CO]{}
	\fancyfoot[CE,CO]{\thepage}
}

\DeclareMathOperator{\C}{\mathbb{C}}

\renewcommand{\i}{\mathbf{i}}
\renewcommand{\tilde}{\widetilde}

\usepackage{fancyhdr}

\fancypagestyle{plain}{%
	\fancyhead[C]{}
	\fancyfoot[C]{}
	
	}
\fancypagestyle{nosection}{%
	\fancyhead[CE]{}
	\fancyhead[CO]{}
	\fancyfoot[CE,CO]{\thepage}
}

\renewcommand{\epsilon}{\varepsilon}

\renewcommand{\i}{\mathbf{i}}
\renewcommand{\tilde}{\widetilde}

\newcommand{\iL}{\mathsf{L}}
\newcommand{\iR}{\mathsf{R}}

\usepackage{tcolorbox}
\usetikzlibrary{tikzmark,calc}

\usepackage{enumerate}
\usepackage{upgreek}

\numberwithin{equation}{section}		
\numberwithin{figure}{section}			
\numberwithin{table}{section}				

\newcommand{\R}{\mathbb{R}}

\newcommand{\A}{\mathcal{A}}

\renewcommand{\O}{\mathcal{O}}

\newcommand{\Rc}{\mathcal{R}}

\newcommand{\Z}{\mathbb{Z}}

\renewcommand{\epsilon}{\varepsilon}

\newcommand{\ds}{\displaystyle}
\newcommand{\iu}{\mathrm{i}\mkern1mu}

\renewcommand{\i}{\mathrm{i}}

\begin{document}
\title[Complex Band Structure]{Complex Band Structure for Subwavelength Evanescent Waves}

\author[Y. De Bruijn]{Yannick De Bruijn}
\address{\parbox{\linewidth}{Yannick De Bruijn\\
Department of Mathematics, ETH Z\"urich, Rämistrasse 101, 8092 Z\"urich, Switzerland.}}
\email{ydebruijn@student.ethz.ch / yannick.debruijn@yale.edu}

\author[E. O. Hiltunen]{Erik Orvehed Hiltunen}
\address{\parbox{\linewidth}{Erik Orvehed Hiltunen\\
Department of Mathematics, University of Oslo, Moltke Moes vei 35, 0851 Oslo, Norway.}}
\email{erikhilt@math.uio.no}

\begin{abstract}
    We present the mathematical and numerical theory for evanescent waves in subwavelength band gap materials. We begin in the one-dimensional case, whereby fully explicit formulas for the complex band structure, in terms of the capacitance matrix, are available. As an example, we show that the gap functions can be used to accurately predict the decay rate of the interface mode of a photonic analogue of the SSH-model. In two dimensions, we derive the band gap Green's function and characterise the subwavelength gap functions via layer potential techniques. By generalising existing lattice-summation techniques, we illustrate our results numerically by computing the complex band structure in a variety of settings.
\end{abstract}
\maketitle
\vspace{3mm}
\noindent
\textbf{Mathematics Subject Classification (MSC2010): 35J05, 35C20, 35P20.}

\vspace{3mm}

\vspace{5mm}

\section{Introduction}
Band gap materials have a fundamental role in the engineering of wave-control devices, due to their ability to act as insulators for waves of certain frequencies. This phenomenon is typically seen in spatially periodic media, whereby waves at band gap frequencies are evanescent (exponentially decaying), and is the key mechanism which underpins the recent developments of wave localisation, guiding, and topological robustness \cite{ammari2018subwavelength,lemoult2016soda,yves2017crytalline,fefferman2014topologically,fefferman2018honeycomb}. A fundamental quantity is the \emph{decay rate} of evanescent waves, which controls the energy leakage of localised or guided waves and can be represented through a complex quasimomenta, giving rise to an extended band structure known as the \emph{complex} band structure \cite{goodwin1939electronic,dang2014complex,reuter2016unified,chang1982complex,tomfohr2002complex}. Although its importance has long been recognised for applications, there are relatively few quantitative mathematical results for computing complex band structures of periodic media, where key attention has instead been on (regular) band structures.  

Recently, the interaction of waves with objects of subwavelength scale has been thoroughly studied due to the ability to overcome the diffraction limit. One way to achieve subwavelength interactions is to use high-contrast metamaterials, which are media constituted by the insertion of a set of highly contrasted resonators into a background medium.  Air bubbles in water was the first demonstrated example of such resonance \cite{minnaert1933musical,ammari2018minnaert,meklachi2018asymptotic}, and it has since been observed in a variety of other settings, such as high-contrast dielectric particles \cite{ammari2019dielectric,ammari2023mathematical}, Helmholtz resonators \cite{ammari2015superresolution}, and plasmonic nanoparticles \cite{ammari2017plasmonicscalar,ammari2016plasmonicMaxwell}. Evanescent waves have been found to contribute towards enhanced near-field transmission, with key applications to subwavelength imaging beyond the diffraction limit (see, for example, \cite{luo2003subwavelength, engelen2009subwavelength,lezec2004diffracted,sukhovich2009experimental,belov2005canalization}).

The field of subwavelength resonant metamaterials has gained significant recent interest, and key phenomena such as band gap opening, distinct topological phases, and non-Hermitian skin effect have been characterised in terms of a discrete capacitance matrix formulation \cite{ammari2023functional}. Although the propagation of waves with frequencies inside the bulk bands has been thoroughly studied, significantly less is known about the propagation of waves at band gap frequencies. Numerical methods are mainly based on Green's functions \cite{chaldyshev1975application,jandieri20191} or plane-wave expansions \cite{EvanescentWaves}.  
In the present work, we introduce a mathematical framework that allows us to characterise the behaviour of waves at band gap frequencies and formally explain the results found in \cite{Edge_Modes, BubbleMeta, BandGapBubbly, doi:10.1137/22M1503841}. We extend the work of \cite{movchan2007band}, wherein a band gap Green's function approach is used to analytically determine the decay length of localised waves, by characterising resonant frequencies of the Helmholtz problem for frequencies inside the band gap. This is achieved by allowing the Floquet-Bloch quasimomentum to become complex. We transform the complex Bloch condition into regular (real) Bloch condition. By introducing a band gap Green's function, we are able to generalise the integral-operator techniques of \cite{ammari2023functional} and characterise the complex band structure as a nonlinear eigenvalue problem for an integral operator. Coupled with the rigorous eigenvalue perturbation techniques of Gohberg and Sigal \cite{gohberg1971operator}, we obtain a capacitance formulation for evanescent waves.

In one dimension, transfer matrix techniques allow simple and explicit formulas for the complex quasimomentum as a function of the frequency. The magnitudes of the eigenvalues of the transfer matrix directly yield the decay rate of the solution to the wave equation. This way, we can find exponentially decaying Bloch modes for frequencies $\omega$ inside the band gap. In the subwavelength regime, this formulation can further be simplified by the capacitance matrix formulation, where the subwavelength band functions are given by the eigenvalues of the (complex) quasiperiodic capacitance matrix. As a concrete application of our method, we consider a subwavelength analogue of the SSH model (originally introduced in \cite{SSH}, see also \cite{lin2022mathematical,ammari2024exponentially}). We show that the complex band structure, which is computed solely in terms of the bulk structure and is independent of the interface, accurately predicts the decay rate of the interface mode.

In a two-dimensional lattice of resonators, the before-mentioned transfer matrix techniques are no longer applicable, and the theory is significantly more challenging. Nevertheless, through a careful application of integral-operator techniqes, we extend the theory of subwavelength resonators to complex quasimomenta and provide robust numerical techniques that allow us to compute the band functions for frequencies inside the band gap, for any direction of decay. In both one- and two-dimensional cases, we illustrate our results numerically. 

This paper is organised as follows. In Section \ref{sec: One dimensional resonator} we treat a one-dimensional resonator system and present explicit formulas for the computation of the band function for frequencies inside the band gap. We derive explicit decay rates  of the gap modes as a function of the frequency. In Section \ref{sec: Two dimensional resonator} we treat a two-dimensional lattice of resonators. We generalise the layer potential theory for solving the boundary problem to the case of complex Bloch quasimomenta and introduce a band gap Green's function. Moreover, we provide a numerical scheme for computing the band functions associated to complex quaimomenta.
In Section \ref{sec: Conclusion} we discuss possible application areas of these new results and future directions of study.


\section{One-dimensional crystals} \label{sec: One dimensional resonator}
In this section, we will explore the propagation of waves inside a one-dimensional resonator chain at band gap frequencies. We will consider two regimes, starting with the subwavelength regime, where we can find explicit expressions for band gap functions. Analytical formulas can be obtained in the general setting using the transfer matrix approach, and we illustrate the results numerically.

\subsection{Setting and problem formulation}
Let us start by introducing the model setting. We consider a one-dimensional infinite chain of resonators \(D := (x_i^{\iL}, x_i^{\iR}\)), where \((x_i^{\iL, \mathrm R})_{i \geq 1} \subset \R\) are the extremities such that $x_i^{\iL} < x_i^{\iR} < x_{i+1}^{\iL}$ for any $i \geq 1$. The lengths of the $i$-th resonators will be denoted by $\ell_i = x_i^{\iR}-x_i^{\iL}$ and the spacings between the $i$-th and the $(i+1)$-th resonators will be denoted by $s_i = x_{i+1}^{\iL}-x_i^{\iR}$. Furthermore, we assume that the resonators are arranged in an $N$-periodic way, that is, $s_{i+N+1} = s_i$, or in other words, the unit cell contains $N$ resonators. In the remainder of this section, we will analyse the cases $N = 1, 2$, which means a single resonator in the unit cell and two resonators inside the unit cell (commonly referred to as a dimer). The resulting system is illustrated in Figure \ref{fig:setting}. The (real) Brillouin zone $Y^*$ is defined as $Y^*:= \R/ (\tfrac{2\pi}{L})\Z$. For band gap modes, the complex Brillouin zone is given by $Y^*+ \iu\R$. 

\begin{figure}[htb]
    \centering
    \begin{adjustbox}{width= 12cm} 
    \begin{tikzpicture}
        \coordinate (x1l) at (1,0);
        \path (x1l) +(1,0) coordinate (x1r);
        \coordinate (s0) at (0.5,0.7);
        \path (x1r) +(0.75,0.7) coordinate (s1);
        
        \path (x1r) +(1.5,0) coordinate (x2l);
        \path (x2l) +(1,0) coordinate (x2r);
        \path (x2r) +(0.5,0.7) coordinate (s2);
        \path (x2r) +(1,0) coordinate (x3l);
        \path (x3l) +(1,0) coordinate (x3r);
        \path (x3r) +(1,0.7) coordinate (s3);
        \path (x3r) +(2,0) coordinate (x4l);
        \path (x4l) +(1,0) coordinate (x4r);
        \path (x4r) +(0.4,0.7) coordinate (s4);
        \path (x4r) +(1,0) coordinate (dots);
        \path (dots) +(1,0) coordinate (x5l);
        \path (x5l) +(1,0) coordinate (x5r);
        \path (x2r) +(1.2,0) coordinate (x6l);
        \path (x5r) +(0.875,0.7) coordinate (s5);
        \path (x6l) +(1,0) coordinate (x6r);
        \path (x6r) +(1.25,0) coordinate (x7l);
        \path (x6r) +(0.57,0.7) coordinate (s6);
        \path (x7l) +(1,0) coordinate (x7r);
        \path (x7r) +(1.5,0) coordinate (x8l);
        \path (x7r) +(0.75,0.7) coordinate (s7);
        \path (x8l) +(1,0) coordinate (x8r);
        \coordinate (s8) at (12,0.7);

        \draw[dotted, line cap=round, line width=1pt, dash pattern=on 0pt off 2\pgflinewidth] ($(x1l) - (1.5, 0)$) -- ($(x1l)- (1.1, 0)$);
        \draw ($(x1l)- (1, 0)$) -- ++(0,1.2);

        \draw[dotted,|-|] ($(x1l)- (1, -0.25)$) -- ($(x1l) + (0, 0.25)$);
        
        \draw[ultra thick] (x1l) -- (x1r);
        \node[anchor=north] (label1) at (x1l) {$x_1^{\iL}$};
        \node[anchor=north] (label1) at (x1r) {$x_1^{\iR}$};
        \node[anchor=south] (label1) at ($(x1l)!0.5!(x1r)$) {$\ell_1$};
        \draw[dotted,|-|] ($(x1r)+(0,0.25)$) -- ($(x2l)+(0,0.25)$);
        \draw[ultra thick] (x2l) -- (x2r);
        \node[anchor=north] (label1) at (x2l) {$x_2^{\iL}$};
        \node[anchor=north] (label1) at (x2r) {$x_2^{\iR}$};
        \node[anchor=south] (label1) at ($(x2l)!0.5!(x2r)$) {$\ell_2$};
        \draw[dotted,|-|] ($(x2r)+(0,0.25)$) -- ($(x3l)+(0,0.25)$);
        \draw[dotted, line cap=round, line width=1pt, dash pattern=on 0pt off 2\pgflinewidth] ($(x3l)+(0.4,0)$) -- ($(x6r)-(0.4,0)$);

        \draw[dotted,|-|] ($(x6r)+(0,0.25)$) -- ($(x7l)+(0,0.25)$);
        \draw[ultra thick] (x7l) -- (x7r);
        \node[anchor=north] (label1) at (x7l) {$x_{N-1}^{\iL}$};
        \node[anchor=north] (label1) at (x7r) {$x_{N-1}^{\iR}$};
        \node[anchor=south] (label1) at ($(x7l)!0.5!(x7r)$) {$\ell_{N-1}$};
        \draw[dotted,|-|] ($(x7r)+(0,0.25)$) -- ($(x8l)+(0,0.25)$);
        \draw[ultra thick] (x8l) -- (x8r);
        \draw[dotted,|-|] ($(x8r)+(0,0.25)$) -- ($(x8r)+(1,0.25)$);
        \draw ($(x8r)+ (1,0)$) -- ++(0,1.2);  

        \draw[<->] ($(x1l) + (-0.8,1)$) -- ($(x8r) + (0.8,1)$) node[midway, above] {$L$};
    
        \draw[dotted, line cap=round, line width=1pt, dash pattern=on 0pt off 2\pgflinewidth] ($(x8r)+(1.1, 0)$) -- ($(x8r) + (1.5, 0)$);
        \node[anchor=north] (label1) at (x8l) {$x_{N}^{\iL}$};
        \node[anchor=north] (label1) at (x8r) {$x_{N}^{\iR}$};
        \node[anchor=south] (label1) at ($(x8l)!0.5!(x8r)$) {$\ell_N$};
        \node[anchor=north] (label1) at (s0) {$s_1$};
        \node[anchor=north] (label1) at (s1) {$s_2$};
        \node[anchor=north] (label1) at (s2) {$s_3$};

        \node[anchor=north] (label1) at (s6) {$s_{N-1}$};
        \node[anchor=north] (label1) at (s7) {$s_{N}$};
        \node[anchor=north] (label1) at (s8) {$s_{N+1}$};
    \end{tikzpicture}
    \end{adjustbox}
    \caption{Unit cell of an infinite chain of subwavelength resonators, with lengths
    $(\ell_i)_{1\leq i\leq N}$ and spacings $(s_{i})_{1\leq i\leq N+1}$.}
    \label{fig:setting}
\end{figure}
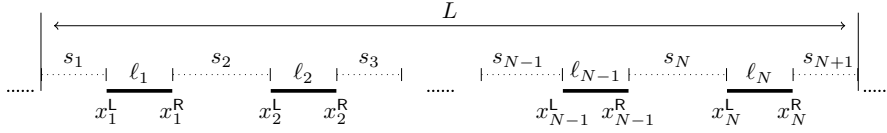

Let $D$ be the domain consisting of the union of all resonators inside $Y$
\begin{equation}
D_i = (x_i^{\iL}, x_i^{\iR}), \quad 
    D = \bigcup_{i=1}^N (x_i^{\iL}, x_i^{\iR}) \subseteq Y.
\end{equation}
We use $v_i$ to denote the wave speed in resonator $D_i$,  while $v$ is the wave speed in the background medium. We consider the Helmholtz scattering problem with transmission boundary conditions across $D$, given by the following system:
\begin{equation}\label{eq: wave equation}
\begin{cases}\frac{\mathrm{d}^2}{\mathrm{d}x^2}u(x)+\frac{\omega^2}{v^2} u=0, & x \in Y \backslash D, \\ \frac{\mathrm{d}^2}{\mathrm{d}x^2}u(x)+\frac{\omega^2}{v_i^2} u=0, & x \in D_i, \\ \left.u\right|_{\iR}\left(x_i^{\iL, \iR}\right)-\left.u\right|_{\iL}\left(x_i^{\iL, \iR}\right)=0, & \text {for all } 1 \leq i \leq N, \\ \left.\frac{\mathrm{d} u}{\mathrm{~d} x}\right|_{\iR}\left(x_i^{\iL}\right)=\left.\delta \frac{\mathrm{d} u}{\mathrm{~d} x}\right|_{\iL}\left(x_i^{\iL}\right), & \text {for all } 1 \leq i \leq N, \\  \left.\frac{\mathrm{d} u}{\mathrm{~d} x}\right|_{\iL}\left(x_i^{\iR}\right)=\left. \delta \frac{\mathrm{d} u}{\mathrm{~d} x}\right|_{\iR}\left(x_i^{\iR}\right), & \text {for all } 1 \leq i \leq N, \\ 
u(x+L) = u(x)e^{\i (\alpha+\i\beta)L}, &\text{for some }\alpha \in Y^*, \beta \in \R.
\end{cases}
\end{equation}
The last condition of \eqref{eq: wave equation} is the Floquet-Bloch condition with \emph{complex} quasimomentum $k= \alpha+\i\beta$ for $\alpha \in Y^*$ and $\beta \in \R$.  The Floquet condition becomes
\begin{align}
    u(x+ L) &= e^{\i (\alpha + \i\beta)L}u(x)\\
    &=e^{-\beta L} e^{\i \alpha L}u(x).
\end{align}
The term $e^{-\beta L}$ governs the decay of the Bloch mode across one unit cell. If $\beta\neq 0$, the solution will be exponentially decaying in one direction. 

The values $\omega = \omega^{\alpha, \beta}$ for which there are nonzero solutions of \eqref{eq: wave equation} are called \emph{resonant frequencies}, or, seen as a function of $\alpha$ with $\beta = 0$, \emph{band functions}. The collection of band functions is called the \emph{band structure}. The complement of the band is called a \emph{band gap}. In the next section, we will establish that resonant frequencies inside the band gap, called \emph{band gap functions} or \emph{gap functions} can be found by allowing $\beta \neq 0$.

\subsection{Subwavelength regime}
The \emph{capacitance matrix} is a well-known tool to study subwavelength resonance (see, for example, \cite{ammari.fitzpatrick.ea2018Mathematical,ammari2023functional,Edge_Modes}). We will revisit the key concepts to characterise subwavelength resonant frequencies and generalise many of those results to the case of complex quaimomenta.
\begin{definition}[Subwavelength resonant frequency]
    Given $\delta > 0$, a subwavelength resonant frequency $\omega = \omega(\delta)$ is defined to be such that
    \begin{enumerate}[(i)]
    \item there exists a nontrivial solution to $(\ref{eq: wave equation})$.
    \item $\omega$ depends continuously on $\delta$ and satisfies $\omega(\delta) \xrightarrow{\delta \to 0} 0$.
    \end{enumerate}
\end{definition}

\subsubsection{Capactiance matrix for complex quasimomenta}
In \cite[Section 4.3]{Edge_Modes}, the subwavelength resonant frequencies of \eqref{eq: wave equation} have been characterised in the case $\beta = 0$. We now introduce a complex quasimomentum $k = \alpha + \i\beta$, for a general $\beta \in \R$, and generalise these results.
\begin{definition}[Quasiperiodic capacitance matrix] Consider solutions $V^{\alpha, \beta}_i$ $:\mathbb{R} \to \R$ of the problem
\begin{equation}\label{eq: chain equation}
    \begin{cases}
        \frac{\mathrm{d}^2}{\mathrm{d}x^2}V^{\alpha, \beta}_i, = 0, & x\in Y \setminus D,\\
        V^{\alpha, \beta}_i(x) = \delta_{ij}, & x \in D_j, \\
        V^{\alpha, \beta}_i(x + mL) = e^{\i (\alpha + \i\beta)mL}V^{\alpha, \beta}_i(x), & m \in \Z,
    \end{cases}
\end{equation}
for $1 \leq i, j \leq N.$ Then the quasiperiodic capacitance matrix is defined coefficientwise by
\begin{equation}\label{eq: cap}
    {C}^{\alpha, \beta}_{ij} = - \int_{\partial D_i} \frac{\partial V^{\alpha, \beta}_j}{\partial \nu}\mathrm{d} \sigma,
\end{equation}
    where $\nu$ is the outward-pointing normal.
\end{definition}

\begin{lemma}
The capacitance matrix for a one-dimensional resonator chain is given by
\begin{equation}\label{def: capacitance_matrix}
    {C}^{\alpha, \beta}_{ij} = -\frac{1}{s_{j-1}}\delta_{i(j-1)} + \left(\frac{1}{s_{j-1}} + \frac{1}{s_j}\right)\delta_{ij} - \frac{1}{s_j}\delta_{i(j+1)} - \delta_{1j}\delta_{iN} \frac{e^{-\i (\alpha + \i\beta) L}}{s_N} - \delta_{1i}\delta_{jN} \frac{e^{\i (\alpha + \i\beta) L}}{s_N}.
\end{equation}
\end{lemma}
\begin{proof}
    The solutions to $(\ref{eq: chain equation})$ for $2 \leq i \leq N-1$ are given by
    \begin{equation}
        V^{\alpha, \beta}_i(x) = \begin{cases}
            \frac{1}{s_{i-1}}(x-x_i^\iL), & x_{i-1}^\iR \leq x \leq x^\iL_i, \\
            1, & x_i^\iL \leq x \leq x_i^\iR,\\
            -\frac{1}{s_i}(x-x^\iL_{i+1}), & x_i^\iR \leq x \leq x_{i+1}^\iL, \\
            0, & \text{else},
         \end{cases}
    \end{equation}
 while for the bigger and smaller $i$ we multiply by the corresponding $e^{\i (\alpha + \i\beta)mL}$-factor. In the one-dimensional case, \eqref{eq: cap} just means
 \begin{equation}
     {C}^{\alpha, \beta}_{ij} = - \left(-\left.\frac{\mathrm{d} V^{\alpha, \beta}_i}{\mathrm{d}x}\right|_\iL(x^\iL_j) + \left. \frac{\mathrm{d} V^{\alpha, \beta}_i}{\mathrm{d}x}\right|_\iR(x_j^\iR) \right).
 \end{equation}
 The result now follows from a direct calculation.
\end{proof}
To take into account the different material parameters inside the resonators, we introduce the generalised capacitance matrix
\begin{equation}\label{eq: generalised capacitance}
    \mathcal{C}^{\alpha, \beta} = W {C}^{\alpha, \beta},
\end{equation}
where the weight matrix $W$ is diagonal and given by $W =  \operatorname{diag}(v_i^2/l_i)$.

In the one-dimensional case, the method of \cite{Edge_Modes} can directly be extended to the case of complex quasimomenta to yield a capacitance-matrix formulation of the subwavelength gap functions. Specifically, we have the following result \cite[Proposition 4.9]{Edge_Modes}.

\begin{proposition}\label{prop: subwavelength band functions}
Assume that the eigenvalues of the generalised capacitance matrix $ \mathcal{C}^{\alpha, \beta}$ are simple. Then, there are $N$ subwavelength band functions $(\alpha, \beta) \mapsto \omega_i^{\alpha, \beta}$ which satisfy
    \begin{equation} \label{eq:sqrt}
        \omega^{\alpha, \beta}_i = \sqrt{\delta \lambda^{\alpha, \beta}_i} + \mathcal{O}(\delta),
    \end{equation}
    where $(\lambda^{\alpha,\beta}_i)_{1 \leq i \leq N}$ are the eigenvalues of $ \mathcal{C}^{\alpha, \beta}$. Here, we choose the branch of \eqref{eq:sqrt} with positive imaginary parts.
\end{proposition}

\subsubsection{Single resonator inside the unit cell}
We begin with the simplest case of a single resonator inside the unit cell $Y$. For simplicity, we take resonator length $l = 1$ and wave speed $v=v_1=1$.  We construct the capacitance matrix using $(\ref{def: capacitance_matrix})$,
\begin{equation}
    {C}^{\alpha, \beta} = \frac{2}{s_1}\bigl(1 - \cos(\alpha L) \cosh(\beta L) - \i\sin(\alpha L) \sinh(\beta L) \bigr).
\end{equation}
Since the band functions have to be real, we find that 
\begin{equation}
    \sin(\alpha L)\sinh(\beta L) = 0.
\end{equation}
This holds if and only if $\alpha = \frac{\pi}{L}n \text{ for some } n \in \Z$ or $\beta = 0$. 

Since the capacitance matrix is a scalar, its eigenvalues correspond to the scalar itself. Using the condition that the band functions have to be real-valued, we find
\begin{align}
    \omega_1^{\alpha, 0} &= \sqrt{\delta \frac{2}{s_1}\bigl(1-\cos(\alpha L)\bigr)},\label{eq: single_real}\\
    \omega_1^{\frac{\pi}{L}, \beta}&= \sqrt{\delta \frac{2}{s_1}\bigl(1+\cosh(\beta L)\bigr)}.\label{eq: single imag}
\end{align}
It is not hard to see that $(\ref{eq: single_real})$ is real-valued for every $\alpha$ in the Brillouin zone while $(\ref{eq: single imag})$ is real for every $\beta \in \R$.

\subsubsection{Resonator dimer inside the unit cell} \label{sec: 1D dimer} We will consider a unit cell containing two resonators (commonly referred to as a \emph{dimer}). For simplicity, we again take resonator length $l = 1$ and denote the spacing between the resonators by $s_1$ and $s_2$, respectively. We also take $v=v_1=v_2=1$. The length of the unit cell is denoted by $L = 2l + s_1 + s_2$.
Following the definition of the quasiperiodic capacitance matrix $(\ref{def: capacitance_matrix})$ for two resonators in the unit cell, we find
\begin{equation}\label{eq: dimer_capacitance}
    {C}^{\alpha, \beta} = \begin{pmatrix}
        \frac{1}{s_1} + \frac{1}{s_2} & -\frac{1}{s_1}- \frac{e^{-\i (\alpha + \i\beta)L}}{s_2} \\
        -\frac{1}{s_1}- \frac{e^{\i (\alpha + \i\beta)L}}{s_2} & \frac{1}{s_1} + \frac{1}{s_2}
    \end{pmatrix}.
\end{equation}
The band functions are characterised via  Proposition $\ref{prop: subwavelength band functions}$. We start by computing the eigenvalues for the dimer capacitance matrix $(\ref{eq: dimer_capacitance})$,
\begin{align}
    \lambda(\alpha, \beta) &=  \frac{1}{s_1} + \frac{1}{s_2} \pm \sqrt{\left( \frac{1}{s_1} + \frac{e^{-\i (\alpha + \i\beta)L}}{s_2}\right)\left(\frac{1}{s_1} + \frac{e^{\i (\alpha + \i\beta)L}}{s_2}\right)}\\
    &=  \frac{1}{s_1} + \frac{1}{s_2} \pm \sqrt{\frac{1}{s_1^2} + \frac{1}{s_2^2} + \frac{2}{s_1s_2}\bigl(\cos(\alpha L)\cosh{(\beta L}) -\i\sin(\alpha L)\sinh(\beta L)\bigr)}.\label{eq: real eigenvalue}
\end{align}
By the same argument as in the previous section, we impose the condition that $\alpha = \frac{\pi}{L}n \text{ for some } n \in \Z$, or $\beta = 0$. We consider both cases separately.
\begin{itemize}
    \item For $\beta = 0$, we recover the ``classical'' band functions with band gaps. The eigenvalues of the capacitance matrix are given by
    \begin{equation}\label{eq: eigenvalue for beta = 0}
        \lambda(\alpha, \beta = 0) = \frac{1}{s_1} + \frac{1}{s_2} \pm \sqrt{\frac{1}{s_1^2} + \frac{1}{s_2^2} + \frac{2}{s_1s_2}\cos(\alpha L)},
    \end{equation}
    therefore, following Proposition \ref{prop: subwavelength band functions}, the two subwavelength band function are given up to order $\delta$ by
    \begin{align}
        \omega_1^{\alpha, 0} &= \sqrt{\delta \left(\frac{1}{s_1} + \frac{1}{s_2} - \sqrt{\frac{1}{s_1^2} + \frac{1}{s_2^2} + \frac{2}{s_1s_2}\cos(\alpha L)} \right)},\\
        \omega_2^{\alpha, 0} &= \sqrt{\delta \left(\frac{1}{s_1} + \frac{1}{s_2} + \sqrt{\frac{1}{s_1^2} + \frac{1}{s_2^2} + \frac{2}{s_1s_2}\cos(\alpha L)} \right)}.
    \end{align}
    \item For $\alpha \in \{0,\pm\frac{\pi}{L}\}$ we have band gap functions when $\beta \neq 0$. Because of $(\ref{eq: real eigenvalue})$, the eigenvalues of the capacitance matrix are of the form 
    \begin{equation}\label{eq: eigenvalue for alpha = pi}
        \lambda(\alpha, \beta) = \frac{1}{s_1} + \frac{1}{s_2} \pm \sqrt{\frac{1}{s_1^2} + \frac{1}{s_2^2} \pm \frac{2}{s_1s_2}\cosh{(\beta L})}.
    \end{equation}
    By Proposition \ref{prop: subwavelength band functions} we find four potential band functions
    \begin{align}
        \omega_1^{\frac{\pi}{L}, \beta} &= \sqrt{\delta \left( \frac{1}{s_1} + \frac{1}{s_2} - \sqrt{\frac{1}{s_1^2} + \frac{1}{s_2^2} - \frac{2}{s_1s_2}\cosh{(\beta L})}\right)} \label{eq: bandGap2},\\
        \omega_2^{\frac{\pi}{L}, \beta} &= \sqrt{\delta \left( \frac{1}{s_1} + \frac{1}{s_2} + \sqrt{\frac{1}{s_1^2} + \frac{1}{s_2^2} - \frac{2}{s_1s_2}\cosh{(\beta L})}\right)} \label{eq: bandGap1},\\
        \omega_3^{0, \beta} &= \sqrt{\delta \left( \frac{1}{s_1} + \frac{1}{s_2} + \sqrt{\frac{1}{s_1^2} + \frac{1}{s_2^2} + \frac{2}{s_1s_2}\cosh{(\beta L})}\right)}\label{eq: bandGap3},\\
        \omega_4^{0, \beta} &= \sqrt{\delta \left( \frac{1}{s_1} + \frac{1}{s_2} - \sqrt{\frac{1}{s_1^2} + \frac{1}{s_2^2} + \frac{2}{s_1s_2}\cosh{(\beta L})}\right)}. \label{eq: bandGap4}
    \end{align}
    To be admissible, these functions must be real-valued.  To see what values of $\beta$ are admissible, we check that the argument of the square root stays positive. For $\omega_1$ and $\omega_2$ this yields the  condition
    \begin{equation}
        \frac{1}{s_1^2} + \frac{1}{s_2^2} - \frac{2}{s_1s_2}\cosh(\beta L) \geq 0, 
    \end{equation}
which holds for all $\beta$ in the following interval:
\begin{equation}\label{eq: beta_domain}
    \beta \in \left[-\frac{1}{L} \operatorname{arcosh}{\left( \frac{1}{2} \frac{s_1^2 + s_2^2}{s_1 s_2}\right)}, \frac{1}{L} \operatorname{arcosh}{\left( \frac{1}{2} \frac{s_1^2 + s_2^2}{s_1 s_2}\right)}\right].
\end{equation}
For $\omega_3$ there is no condition on $\beta$ and, as a consequence, the band function is well-defined for any $\beta \in \R$. For $\omega_4$, we find that the expression is imaginary for any $\beta \neq 0$, thus not corresponding to a gap function.
We would like to highlight that the interval \eqref{eq: beta_domain} for $\beta$ does not depend on the contrast $\delta$. In other words, the decay inside the first band gap is bounded by a universal constant (given by $(\ref{eq: beta_domain})$) in the limit $\delta \to 0$. 
\end{itemize}

\subsubsection{Numerical computation}
In this section, we numerically illustrate the formulas for the band functions that we found in the previous section.
As expected for frequencies inside the band gap, we can always find a corresponding band gap function with $\beta \neq 0$.
\begin{figure}[h]
    \centering
    \subfloat[][Band structure for a single resonator inside the unit cell. Computation performed for $s_1 = 0.6$ and $\delta = 0.1$.]{{\includegraphics[width=6.5cm]{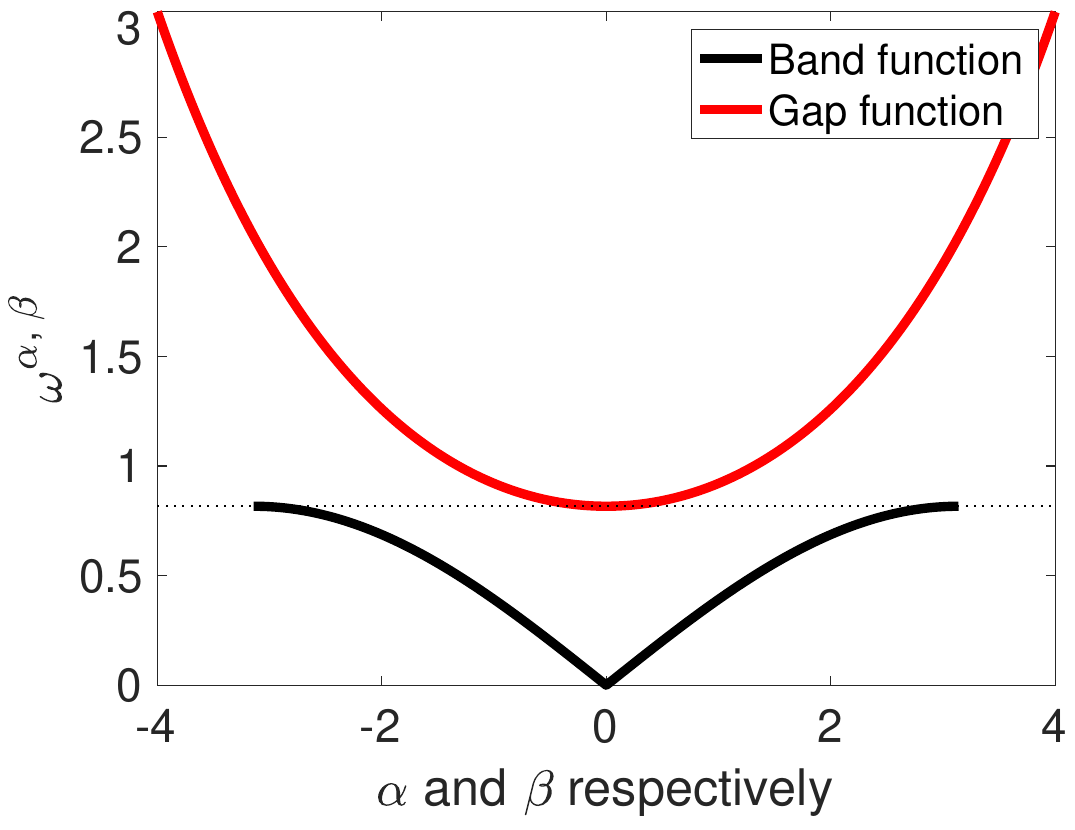} }}\qquad
    \subfloat[][Band structure for a dimer inside the unit cell. Computation performed with $s_1 = 0.8, s_2 = 2$ and $\delta = 0.1$. ]{{\includegraphics[width=6.5cm]{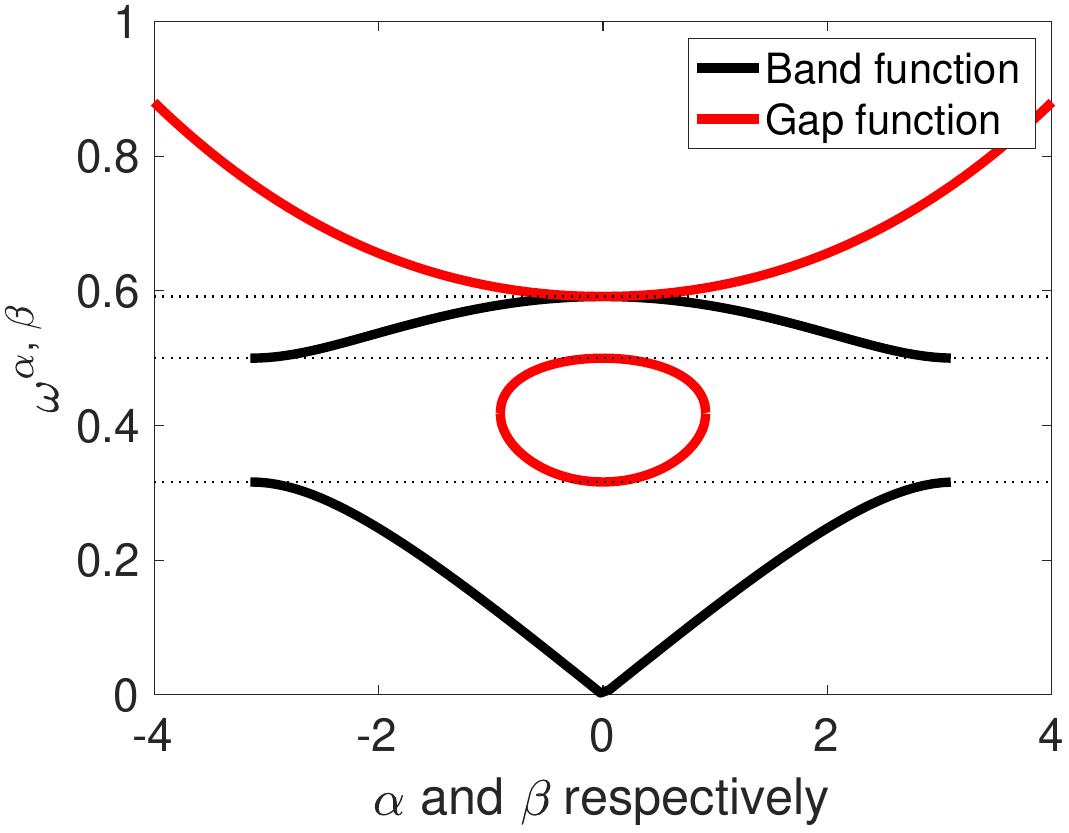} }}
    \caption{The band functions $\omega^{\alpha, 0}_i$ have been drawn over the first Brillouin zone, whereas the respective band gap functions $\omega^{\alpha, \beta}_i$ have been drawn over the real space or the interval $(\ref{eq: beta_domain})$.}
   \label{Fig: spectrum_dimer_system}
\end{figure}
At this point, it is important to emphasise that the band and band gap function cover the entire subwavelength frequency range. In particular, for any arbitrary frequency $\omega^*$ there exists a pair $\alpha$ and $\beta$ such that $\omega(\alpha, \beta) = \omega^*$.

\subsubsection{Application to exponentially localised interface modes}

In this section, we analyse a finite chain of dimers with a geometric defect at the centre. The eigenmode corresponding to the defect frequency displays an exponential decay away from the defect interface in both directions, commonly referred to as an interface mode. The structure can be seen as an analogue of the well-studied SSH model, and the existence of an interface mode was established in \cite{ammari2024exponentially}. Our goal is to show that the decay rate is governed by the complex quasimomentum associated to the gap eigenfrequency computed in \Cref{sec: 1D dimer}.

\begin{figure}[h]
    \centering
    \begin{adjustbox}{width=\textwidth}
        \begin{tikzpicture}
        \draw[-,thick,dotted] (-.5,-1) -- (-.5,2);
        \draw[|-|,dashed] (0,1) -- (1,1);
        \node[above] at (0.5,1) {$s_2$};
        \draw[ultra thick] (1,0) -- (2,0);
        \node[below] at (1.5,0) {$D_{2m+2}$};
        \draw[-,dotted] (1,0) -- (1,1);
        
        \draw[|-|,dashed] (2,1) -- (2.5,1);
        \node[above] at (2.25,1) {$s_1$};
        \draw[ultra thick] (2.5,0) -- (3.5,0);
        \node[below] at (3,0) {$D_{2m+3}$};
        \draw[-,dotted] (2.5,0) -- (2.5,1);
        \draw[-,dotted] (3.5,0) -- (3.5,1);
        \draw[-,dotted] (2,0) -- (2,1);
        
        \begin{scope}[shift={(+4,0)}]
        \draw[|-|,dashed] (-0.5,1) -- (0.5,1);
        \node[above] at (0,1) {$s_2$};
        \draw[ultra thick] (0.5,0) -- (1.5,0);
        \node[below] at (1,0) {$D_{2m+4}$};
        \draw[-,dotted] (0.5,0) -- (0.5,1);
        \node at (2.5,.5) {\dots};
        \end{scope}

        \begin{scope}[shift={(+6,0)}]
        \draw[ultra thick] (1,0) -- (2,0);
        \node[below] at (1.5,0) {$D_{4m-2}$};
        
        \draw[|-|,dashed] (2,1) -- (3,1);
        \node[above] at (2.5,1) {$s_2$};
        \draw[ultra thick] (3,0) -- (4,0);
        \node[below] at (3.5,0) {$D_{4m}$};
        \draw[-,dotted] (2,0) -- (2,1);
        \draw[-,dotted] (3,0) -- (3,1);
        \draw[-,dotted] (4,0) -- (4,1);
\begin{scope}[shift={(+2,0)}]
        \draw[|-|,dashed] (2,1) -- (2.5,1);
        \node[above] at (2.25,1) {$s_1$};
        \draw[ultra thick] (2.5,0) -- (3.5,0);
        \node[below] at (3,0) {$D_{4m+1}$};
        \draw[-,dotted] (2.5,0) -- (2.5,1);
        \end{scope}
        \end{scope}
        
\begin{scope}[shift={(-4,0)}]

        \draw[|-|,dashed] (0.5,1) -- (1,1);
        \node[above] at (0.75,1) {$s_1$};
        \draw[ultra thick] (1,0) -- (2,0);
        \node[below] at (1.5,0) {$D_{2m}$};
        \draw[-,dotted] (1,0) -- (1,1);
        
        \draw[|-|,dashed] (2,1) -- (3,1);
        \node[above] at (2.5,1) {$s_2$};
        \draw[ultra thick] (3,0) -- (4,0);
        \node[below,fill=white] at (3.5,0) {$D_{2m+1}$};
        \draw[-,dotted] (2,0) -- (2,1);
        \draw[-,dotted] (3,0) -- (3,1);
        \draw[-,dotted] (4,0) -- (4,1);
        \end{scope}
        
        \begin{scope}[shift={(-8,0)}]
        \node at (.5,.5) {\dots};
        \draw[ultra thick] (1.5,0) -- (2.5,0);
        \node[below] at (2,0) {$D_{2m-2}$};
        
        \draw[|-|,dashed] (2.5,1) -- (3.5,1);
        \node[above] at (3,1) {$s_2$};
        \draw[ultra thick] (3.5,0) -- (4.5,0);
        \node[below] at (4,0) {$D_{2m-1}$};
        \draw[-,dotted] (2.5,0) -- (2.5,1);
        \draw[-,dotted] (3.5,0) -- (3.5,1);
        \draw[-,dotted] (4.5,0) -- (4.5,1);
        \end{scope}

        \begin{scope}[shift={(-14,0)}]
        \draw[ultra thick] (1,0) -- (2,0);
        \node[below] at (1.5,0) {$D_{1}$};
        
        \draw[|-|,dashed] (2,1) -- (2.5,1);
        \node[above] at (2.25,1) {$s_1$};
        \draw[ultra thick] (2.5,0) -- (3.5,0);
        \node[below] at (3,0) {$D_{2}$};
        \draw[-,dotted] (2,0) -- (2,1);
        \draw[-,dotted] (2.5,0) -- (2.5,1);
        \draw[-,dotted] (3.5,0) -- (3.5,1);
\begin{scope}[shift={(+2,0)}]
        \draw[|-|,dashed] (1.5,1) -- (2.5,1);
        \node[above] at (2,1) {$s_2$};
        \draw[ultra thick] (2.5,0) -- (3.5,0);
        \node[below] at (3,0) {$D_{3}$};
        \draw[-,dotted] (2.5,0) -- (2.5,1);
        \end{scope}
        \end{scope}
        
        \end{tikzpicture}
    \end{adjustbox}
    \caption{Dimer structure with a geometric defect.}
    \label{fig: geometrical defect}
\end{figure}
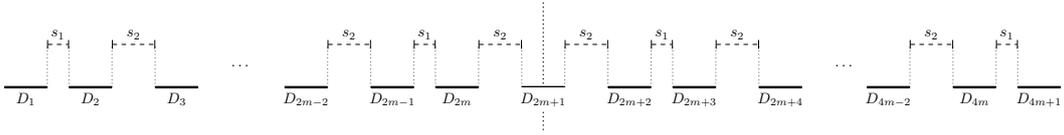

The subwavelength resonances inside the resonator chain may be computed using the capacitance matrix.
The authors of \cite{ammari2024exponentially} have derived an asymptotic formula for the bandgap resonant frequency. 
\begin{theorem}[{\cite[Theorem 5.5]{ammari2024exponentially}}]\label{thm:existenceofeigenfrquency}
    Consider a perturbed structure of dimers as illustrated in \cref{fig: geometrical defect}. For $N$ large enough, there exists a unique interface mode with eigenfrequency $\omega_{\mathsf{i}}^{(N)}$ in the band gap. The associated eigenfrequency $\omega_{\mathsf{i}}^{(N)}$ converges to 
    \begin{align}\label{eq: defect resonance}
        \omega_{\mathsf{i}} = \sqrt{\delta \frac{1}{2} \left(-\sqrt{\frac{9}{s_1^2}- \frac{14}{s_1s_2} + \frac{9}{s_2^2}}+\frac{3}{s_1}+\frac{3}{s_2}\right)}
    \end{align}
    exponentially as $N\to\infty$. In particular, for $N$ big enough,
    \begin{align}\label{eq: error estimate convergence frequency in gap}
        \vert \omega_{\mathsf{i}} - \omega_{\mathsf{i}}^{(N)}\vert 
        < Ae^{-BN},
    \end{align}
    for some $A,B>0$ independent of $N$.
\end{theorem}
Combining this result with our calculations of the complex band structure in \Cref{sec: 1D dimer}, we are now able to determine the decay rate of the interface mode.
\begin{corollary}
    Consider a perturbed structure of dimers as illustrated in \cref{fig: geometrical defect}. For $N$ large enough, there exists a unique interface mode with eigenfrequency $\omega_{\mathsf{i}}^{(N)}$ in the band gap. The decay rate of the interface mode converges to
    \begin{equation}
        \beta = \frac{1}{L}\operatorname{arcosh}\left(\frac{s_1 s_2}{2}\left(\left(\frac{ \omega_{\mathsf{i}}^2}{\delta}-\frac{1}{s_1}-\frac{1}{s_2}\right)^2 -\frac{1}{s_1^2}-\frac{1}{s_2^2}\right)\right)
    \end{equation}
    exponentially as $N\to\infty$, where $ \omega_{\mathsf{i}}$ is given by \eqref{eq: defect resonance}.
\end{corollary}
\begin{proof}
    The result follows directly from solving \eqref{eq: bandGap3} for $\beta$.
\end{proof}

\begin{figure}[h]
    \centering
    \subfloat[][Band structure for a dimer in the unit cell with the resonant frequencies. The gap frequency has a complex quasimomentum $\beta = 0.1154$.]{{\includegraphics[width=8cm]{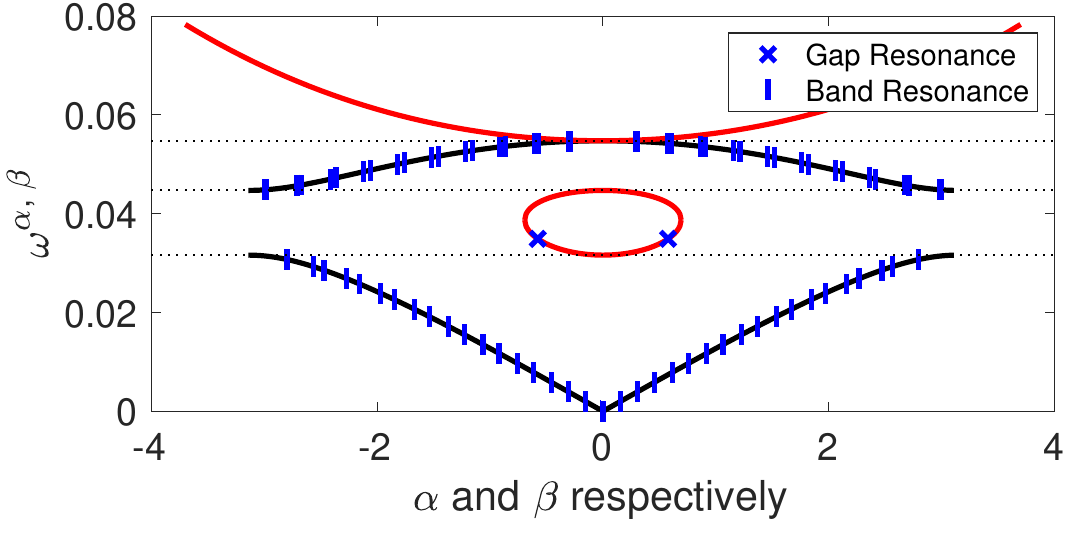} }}\\
    \subfloat[][Eigenmode (blue) associated to the gap resonance, contained in the decay envelope (red) predicted by the complex quasimomentum.]{{\includegraphics[width=6.67cm]{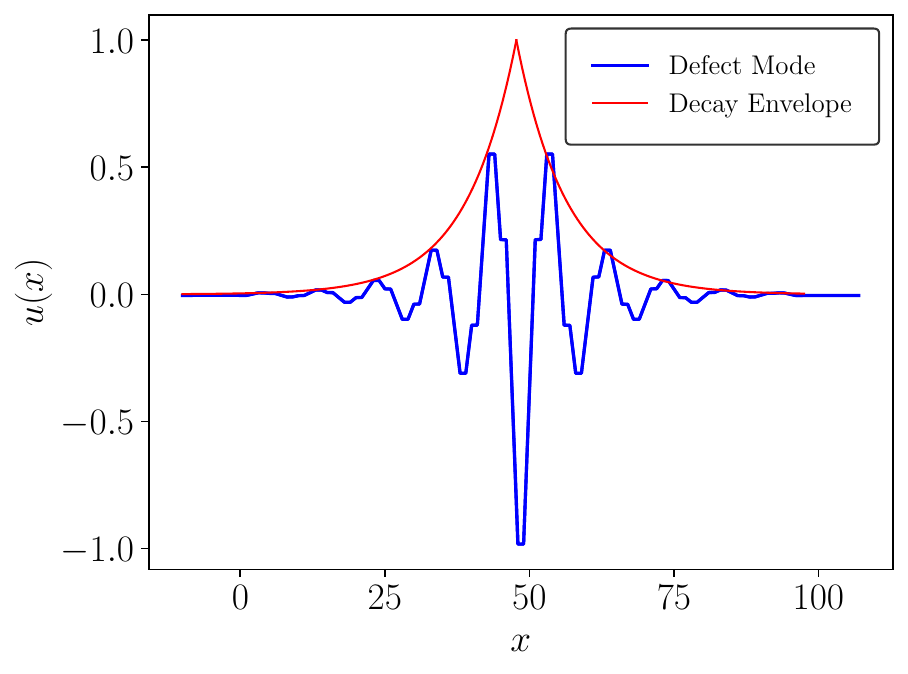} }}\quad
    \subfloat[][Same figure as (B) but with a logarithmic $y$-scale. The exponential decay is correctly predicted by the complex quasimomentum.]{{\includegraphics[width=6.67cm]{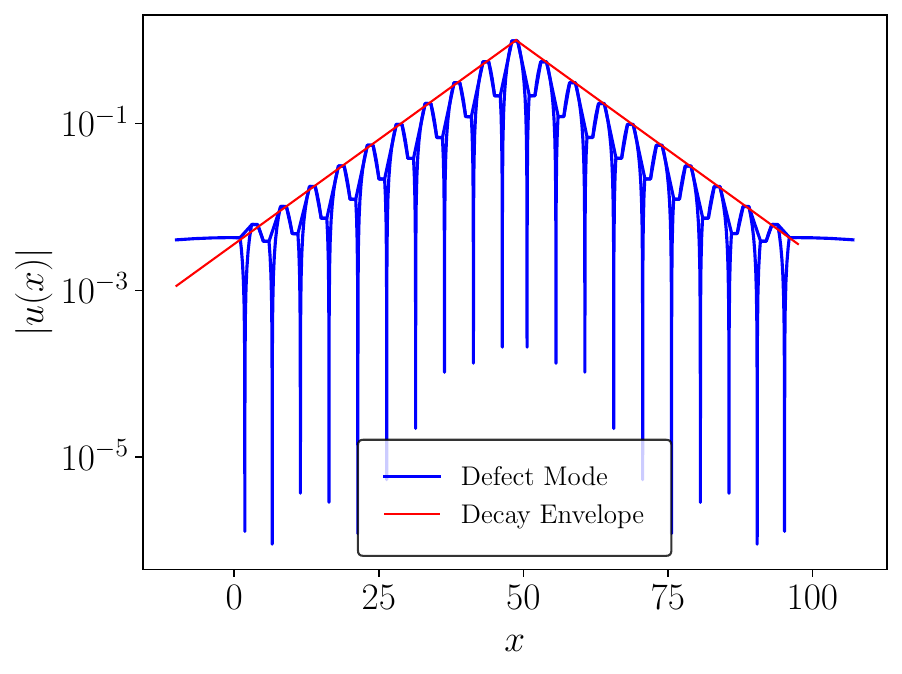} }}
    \caption{The complex band structure accurately predicts the decay rate of localised modes. Here, we consider the discrete band structure of a finite SSH chain (A). The structure has an interface mode (B), which is exponentially localized and whose decay rate is governed by the complex band structure (C). Here, the computations are  performed for $41$ dimers, $s_1 = 1, s_2 = 2$, and $\delta = 0.001$.}
   \label{Fig: Eigenmode loclization}
\end{figure}
Figure \ref{Fig: Eigenmode loclization} shows that the complex quasimomentum associated with the defect eigenfrequency accurately predicts the decay length of the defect mode.

\subsection{General frequency regime}\label{sec:general}
We now consider a general (not necessarily subwavelength) frequency regime. We begin with a general symmetry result, which shows that the gap functions pin $\alpha$ to the centre or edge of the Brillouin zone.
\begin{theorem}
    \label{Thm: alpha in the band gap}
    Let $\omega(\alpha,\beta)$ be a frequency inside the band gap $\beta \neq 0$. Then $\alpha$ is either equal to $0$ or fixed at the edge of the Brillouin zone, i.e., $\alpha \in \{0, \pm\frac{\pi}{L}\}$.
\end{theorem}
\begin{proof}
    We argue by symmetry. Note that when $u(x)$ is a solution to the scattering problem $(\ref{eq: wave equation})$ for some $\alpha$ and $\beta$, then $\overline{u}(-x)$ and $\overline{u}(x)$ are also solutions, and must therefore satisfy the Floquet condition
    \begin{align}
        \overline{u}(-x-L) &= \overline{e^{-\i (\alpha + \i\beta)L}}\overline{u}(-x) \\
        &= e^{\i (\alpha - \i\beta)L} \overline{u}(-x),
    \end{align}
    as well as
    \begin{align}
        \overline{u}(x + L) &= \overline{e^{\i (\alpha + \i\beta)L}}\overline{u}(x)\\
        &=e^{\i (-\alpha + \i\beta)L}\overline{u}(x).
    \end{align}
    In the one-dimensional case, the band functions are monotone and therefore nondegenerate. In other words, for each $\omega$ there exist at most 2 pairs $(\alpha_1, \beta_1)$ and $(\alpha_2, \beta_2)$. 
    Let $\omega$ be a band gap frequency, then $\beta \neq 0$. If $(\omega, \alpha, \beta)$ is a solution, then $(\omega, \alpha, -\beta)$ and $(\omega, -\alpha, \beta)$ must also be solutions. As a consequence of the nondegeneracy, it must be that $-\alpha = \alpha$. Because $\alpha$ is periodic with respect to the Brillouin zone, it follows that either $\alpha = 0$ or $\alpha = \pm \frac{\pi}{L}$, which corresponds to the edge of the Brillouin zone.
\end{proof}

\begin{remark}
    The same argument from Corollary \ref{Thm: alpha in the band gap} can be repeated for $\beta$ inside the bulk bands. The nondegeneracy condition now gives  $\beta = - \beta$, so that $\beta = 0$. This agrees with the classical Floquet-Bloch theory: inside the bulk bands, the modes propagate without decay.
\end{remark}

In the case of a one-dimensional resonator chain, we can explicitly derive the solution to $(\ref{eq: wave equation})$. The transfer matrix formalism (see, for example, \cite{Transfer_Matrix_Griffiths} for an overview) allows us to analyse how wave amplitudes of the general solution are transformed once they pass through a single resonator. We emphasise that this formalism is already well established; nevertheless, for the convenience of the reader, we recall the main results and proofs on the structure of the transfer matrix. The general solution to the one dimensional scattering problem $(\ref{eq: wave equation})$, in the region exterior to the resonators, can be written
\begin{equation}\label{eq:uPW}
    u(x) = \begin{cases}
        Ae^{\i kx} + Be^{-\i kx}, & \text{ if } x < x_1^\iL,  \\
        Ce^{\i kx} + De^{-\i kx}, & \text{ if } x > x_N^\iR,
    \end{cases}
\end{equation}
where the wave number is given by $ k := \frac{\omega}{v}$.

Due to the linearity of the Helmholtz equation, we have a linear relation between the amplitudes of left- and right-going waves. The transfer matrix $M$ links the amplitudes $(A, B)$ to the left of the resonator to the amplitudes $(C, D)$ to the right via
\begin{equation}\label{eq:M}
    M\begin{pmatrix}
        A \\
        B
    \end{pmatrix} = \begin{pmatrix}
        C \\ D
    \end{pmatrix},\quad M \in \R^{2 \times 2},
\end{equation}
where
\begin{equation}\label{eq: def of transfer matrix}
    M := \begin{pmatrix}
        M_{11} & M_{12} \\ M_{21} & M_{22}
    \end{pmatrix}.
\end{equation}
Passing over $i$ unit cells then corresponds to the transfer matrix raised to the $i$-th power. 

\begin{lemma}\label{lemma: transfer matrix structure}
    Let $M \in \R^{2 \times 2}$ be the transfer matrix, $\det(M) = 1$ and $M$ has the form
    \begin{equation}\label{eq:Ma}
        M = \begin{pmatrix}
            M_{11} & M_{12} \\
            M_{12}^* & M_{11}^*
        \end{pmatrix}.
    \end{equation}
\end{lemma}

\begin{proof}
    The determinant of $M$ being equal to $1$ is a direct  consequence of the conservation of energy. To show \eqref{eq:Ma}, we argue by symmetry. Let $u(x)$ be a solution to $(\ref{eq: wave equation})$, then by time-reversal symmetry $\overline{u}(x)$ is also a solution. Therefore, since the direction of propagation of the modes is reversed,
    \begin{equation}\label{eq:M*}
        M\begin{pmatrix}
            B^* \\ A^*
        \end{pmatrix} = \begin{pmatrix}
            D^* \\ C^*
        \end{pmatrix},
    \end{equation}
    which is equivalent to
    \begin{equation}
         \begin{pmatrix}
            M_{22}^* & M_{21}^* \\
            M_{12}^* & M_{11}^*
        \end{pmatrix} \begin{pmatrix}
            A \\ B
        \end{pmatrix} =\begin{pmatrix}
            C \\ D
        \end{pmatrix}.
    \end{equation}
    Comparing this to \eqref{eq: def of transfer matrix} we have $M_{11} = M_{22}^*$ as well as $M_{21} = M_{12}^*$ and the assertion follows.
\end{proof}

\begin{lemma}\label{lemma: eigenvalues transfer}
The two eigenvalues $\lambda_1$ and $\lambda_2$ of the transfer matrix $M$ are either real and inverses of each other, or nonreal and complex conjugates of each other. 
\end{lemma}

\begin{proof}
    Following Lemma \ref{lemma: transfer matrix structure} the trace of the transfer matrix is real, so we may define the eigenvalues as $\lambda_1 = x + \i y $ and $\lambda_2 = z - \i y$.
    Lemma \ref{lemma: transfer matrix structure} implies that the product of the eigenvalues $\lambda_1 \lambda_2 = 1$:
    \begin{align}
        xz + y^2+\i y(z-x) &= 1.
    \end{align}
Consequently, $y = 0$ or $x= z$. If $y=0$, the eigenvalues are real, and we know from \Cref{lemma: transfer matrix structure} that
\begin{equation}
    \lambda_1 = \frac{1}{\lambda_2}.
\end{equation}
If $x=z$, both eigenvalues are complex conjugates of each other, and again by \Cref{lemma: transfer matrix structure}, they lie on the unit circle:
\begin{align}
    \lambda_1 = e^{\i \varphi}, \quad \text{and} \quad  \lambda_2 = e^{-\i \varphi},
\end{align}
for some $\varphi$.
\end{proof}

Let $L$ be the length of the unit cell. At the edge of the unit cell, the complex Floquet condition corresponds to
\begin{align}
    u( x + L) = e^{\i (\alpha + \i\beta)L}u(x), \quad \text{and} \quad 
    u'( x + L) = e^{\i (\alpha + \i\beta)L}u'(x),
\end{align}
which is equivalent to the matrix equation
\begin{align} \label{eq: floquet_transfer}
    \begin{pmatrix}
        C \\ D
    \end{pmatrix}
    = e^{\i (\alpha + \i\beta)L} \begin{pmatrix}
        e^{-\i kL} & 0 \\
        0 & e^{\i kL}
    \end{pmatrix}\begin{pmatrix}
        A \\ B
    \end{pmatrix}.
\end{align}
This gives us a second equation which links the amplitudes at the beginning and the end of the unit cell. Rearranging terms leads to the eigenvalue problem,
\begin{equation}\label{eq: modifier transfer eigenvalue problem}
    \underbrace{ \begin{pmatrix}
        e^{\i kL} & 0 \\
        0 & e^{-\i kL}
    \end{pmatrix}M}_{ =: T(k)}\begin{pmatrix}
        A \\ B
    \end{pmatrix} = e^{\i (\alpha + \i\beta)L}\begin{pmatrix}
        A \\ B
    \end{pmatrix}.
\end{equation}
We will call $T(k)$ the modified transfer matrix. It is not hard to see that the properties from Lemma \ref{lemma: transfer matrix structure} and therefore Lemma \ref{lemma: eigenvalues transfer} are conserved.

\begin{theorem}\label{thm: master theorem}
    Let $k$ be an arbitrary frequency, and let $\lambda_1(k)$ and $\lambda_2(k)$ be the eigenvalues of the modified transfer matrix $T(k)$ for some $L>0$. Then the eigenmode $u_i$, for $i=1,2,$ propagates away from the origin as
    \begin{equation}
        u_i(x + L) = e^{\i (\alpha(k) + \i\beta(k)) L}u_i(x), \quad L \in \R,
    \end{equation}
    with
   \begin{equation}
    \begin{cases}
        \beta(k) = \frac{-1}{L}\log(\lvert \lambda_i(k) \rvert ), \\
        \alpha(k) = \frac{1}{L}\arg( \lambda_i(k)).
    \end{cases}
\end{equation}
\end{theorem}

\begin{proof}
    Let $\lambda_1(k)$ and $\lambda_2(k)$ be the eigenvalues of $T(k)$ introduced in $(\ref{eq: modifier transfer eigenvalue problem})$. We may write
\begin{equation}
    \lambda_i(k) = e^{\i (\alpha + \i\beta)L}.
\end{equation}
Taking the complex logarithm yields
\begin{align}
    \i L \alpha  - L\beta &= \log(\lambda_i(k)) \\
    &= \log(\lvert \lambda_i(k) \rvert ) + \i \arg(\lambda_i(k)).
\end{align}
Comparing their respective real and imaginary parts yields the desired explicit formulas for $\alpha$ and $\beta$.
\end{proof}

We proceed by explicitly computing the transfer matrix associated to a single resonator (i.e., $N=1$) of length $2a$ centered inside a unit cell of length $L$ (although the method immediately generalises to multiple resonators). In this case, the solution is given by
\begin{equation}\label{eq:uPW2}
    u(x) = \begin{cases}
        Ae^{\i kx} + Be^{-\i kx}, & \text{ if } x <-a,  \\
        Fe^{\i k'x} + Ge^{-\i k'x}, & \text{ if } -a\leq x \leq a, \\
        Ce^{\i kx} + De^{-\i kx}, & \text{ if } x > a,
    \end{cases}
\end{equation}
where the wave number $k'$ is given by $k' := \frac{\omega}{v_1}$. We also introduce $n$ such that $k' = n k$, with frequency coupling constant $n = v_1/v$ specific to the system. When traversing the first boundary in $x = -a$ we obtain,
\begin{equation}
\begin{cases}
    Ae^{-\i ka} + Be^{\i ka} = Fe^{-\i k'a} + Ge^{\i k'a}, & \text{ continuity of field at } -a, \\
    \delta \i k(Ae^{-\i ka}-Be^{\i ka}) = \i k'(Fe^{-\i k'a}-Ge^{\i k'a}), & \text{ continuity of flux at } -a,
\end{cases}
\end{equation}
which is equivalent to the matrix formulation,
\begin{equation}
    \begin{pmatrix}
        1 &1 \\ \delta k & -\delta k
    \end{pmatrix}\begin{pmatrix}
        Ae^{-\i ka} \\ Be^{\i ka}
    \end{pmatrix} = \begin{pmatrix}
        1 &1 \\ k' & -k'
    \end{pmatrix}\begin{pmatrix}
        Fe^{-\i k'a} \\ Ge^{\i k'a}
    \end{pmatrix}.
\end{equation}
By rearranging the matrices we may write the scattered wave functions as a function of the incident waves,
\begin{equation}\label{eq: AB to FG}
    \begin{pmatrix}
        F \\ G
    \end{pmatrix} = \frac{1}{2} \begin{pmatrix}
        e^{\i k'a} &0 \\  0 & e^{-\i k'a}     \end{pmatrix}\begin{pmatrix}
        1 + \frac{k \delta}{k'} & 1-\frac{k \delta}{k'}\\
        1-\frac{k \delta}{k'}& 1+\frac{k \delta}{k'}
    \end{pmatrix}\begin{pmatrix}
        e^{-\i ka} & 0 \\ 0& e^{\i ka}
    \end{pmatrix}\begin{pmatrix}
        A \\ B
    \end{pmatrix} .
\end{equation}
This expression allows us to express the amplitudes of $F$ and $G$ as functions of $A$ and $B$.
Proceeding in the same way, taking into account the boundary conditions at $x = a$, we find that
\begin{equation}\label{eq: FG to CD}
    \begin{pmatrix}
        C \\ D    \end{pmatrix} = \frac{1}{2} \begin{pmatrix}
        e^{-\i ka} & 0 \\ 0 & e^{\i ka}
    \end{pmatrix}\begin{pmatrix}
        1 + \frac{k'}{k\delta} & 1- \frac{k'}{k\delta}\\
        1-\frac{k'}{k\delta} & 1 + \frac{k'}{k\delta}
    \end{pmatrix} \begin{pmatrix}
        e^{\i k'a} & 0 \\
        0 & e^{-\i k'a} 
    \end{pmatrix}\begin{pmatrix}
        F \\ G
    \end{pmatrix}.
\end{equation}
As a consequence, equations \eqref{eq: AB to FG} and \eqref{eq: FG to CD}  allow us to write the right amplitudes waves as a function of the left amplitudes,
{\begin{align}
    \begin{pmatrix}
        C \\ D
    \end{pmatrix}
    &=\underbrace{\begin{pmatrix}
\frac{e^{-2 \i a k} \left(2 \delta n \cos(2 a k n) + \i (\delta^2 + n^2) \sin(2 a k n)\right)}{2 \delta n} & -\frac{\i (\delta - n) (\delta + n) \sin(2 a k n)}{2 \delta n} \\
\frac{\i (\delta - n) (\delta + n) \sin(2 a k n)}{2 \delta n} & \frac{e^{2 \i a k} \left(2 \delta n \cos(2 a k n) - \i (\delta^2 + n^2) \sin(2 a k n)\right)}{2 \delta n}
\end{pmatrix}}_{= M}\begin{pmatrix}
        A \\ B    \end{pmatrix}. \label{eq: transfer matrix square potential}
\end{align}}
We have effectively linked the wave amplitudes before the resonator to those after the resonator via the transfer matrix.
Following \eqref{eq: modifier transfer eigenvalue problem}, the resulting eigenvalue problem is
{\small \begin{multline}\label{eq: transfer_Matrix_eigenvalue_problem}
    \overbrace{\begin{pmatrix}
\frac{e^{-\i k (2 a - L)} \left(2 \delta n \cos(2 a k n) + \i (\delta^2 + n^2) \sin(2 a k n)\right)}{2 \delta n} & -\frac{\i e^{-\i k L} (\delta - n) (\delta + n) \sin(2 a k n)}{2 \delta n} \\
\frac{\i e^{\i k L} (\delta - n) (\delta + n) \sin(2 a k n)}{2 \delta n} & \frac{e^{-\i k (-2 a + L)} \left(2 \delta n \cos(2 a k n) - \i (\delta^2 + n^2) \sin(2 a k n)\right)}{2 \delta n}
\end{pmatrix}
}^{= T(k)}\begin{pmatrix}
    A \\ B
\end{pmatrix}\qquad\\ \hspace{50mm}= e^{\i (\alpha + \i\beta)L} \begin{pmatrix}
    A \\ B
\end{pmatrix}.
\end{multline}}\\%
While the eigenvalues of the matrix $T(k)$ could be computed explicitly, we refrain from doing so due to the resulting lengthy expressions. Instead, we illustrate the results numerically.

Figure \ref{Fig: General Regime} (A)  depicts Theorem \ref{Thm: alpha in the band gap} and one observes that for band gap frequencies, $\alpha$ is either $0$ or at the edge of the Brillouin Zone. In other words, $\alpha$ is constant for band gap frequencies. Moreover, $\lvert \beta \rvert$ is bounded, and the maximal decay occurs in the middle of the band gap.

Figure \ref{Fig: General Regime} (B) illustrates that asymptotically, we may recover the subwavelength regime from the general regime by letting $\delta \to 0$. As seen, the width of the first gap increases as $\delta \to 0$, while the subwavelength frequency regime is approximated by \Cref{Fig: spectrum_dimer_system} (A). 

\begin{figure}[tbh]
    \centering
    \subfloat[][Computation performed for $n_0 = 1.8, a = 0.2, b = 0.5$ and $\delta = 1$. ]{{\includegraphics[height=5.2cm]{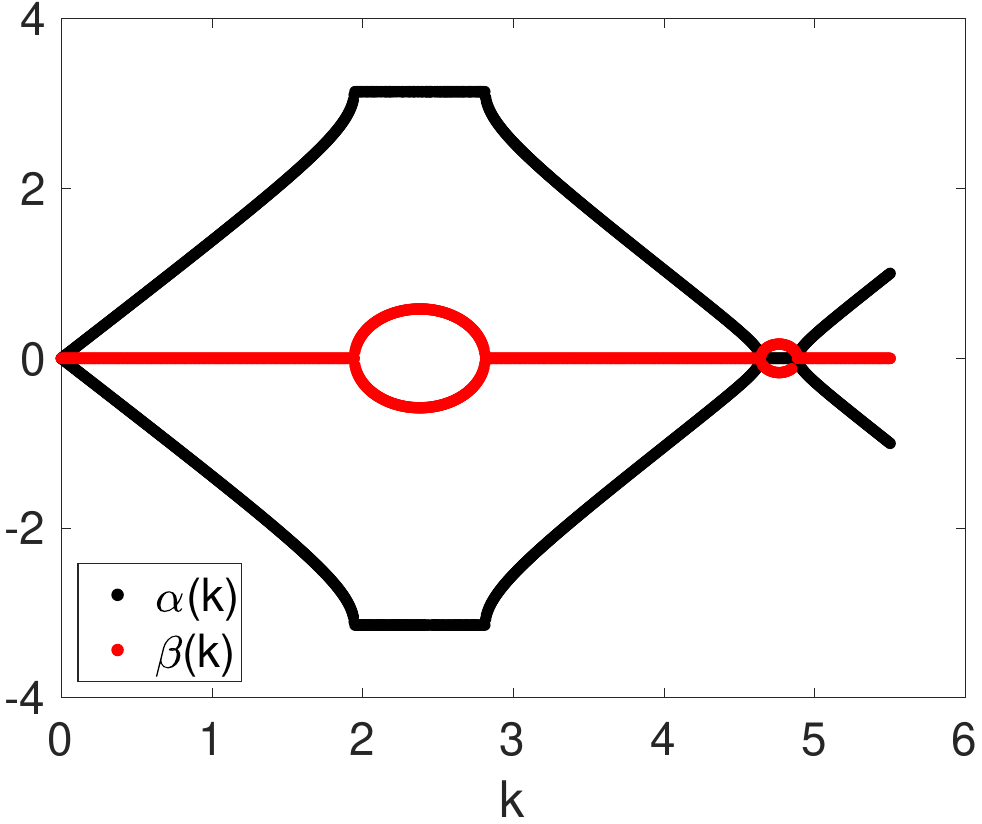} }}
    \qquad
    \subfloat[][Computation performed for $n_0 = 1.8, a = 0.2, b = 0.5$ and $\delta = 0.05$.]{{\includegraphics[height=5.2cm]{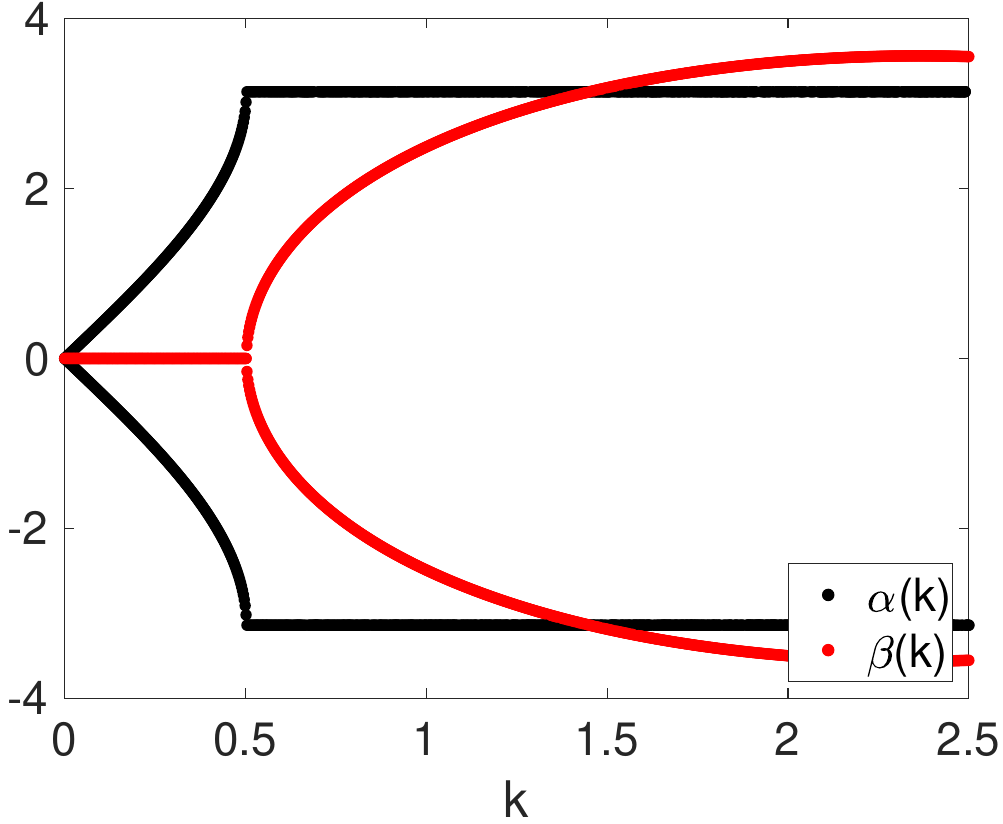} }}%
    \caption{Plot for $\alpha(k)$ and $\beta(k)$, following Theorem \ref{thm: master theorem}. In Figure (B) we illustrate that the subwavelength regime may be recovered from the General Regime by making the contrast small enough.}
   \label{Fig: General Regime}
\end{figure}


\section{Two-dimensional crystals}\label{sec: Two dimensional resonator}

In this section, we will consider a two-dimensional periodic structure of resonators. In the two-dimensional case, the transfer matrix approach of \Cref{sec:general} is no longer applicable, and we instead develop quasiperiodic layer-potential techniques to find evanescent solutions to the scattering problem and characterise the band gap functions.

\subsection{Setting and problem formulation} We consider a two-dimensional crystal of subwavelength resonators, repeated periodically in a lattice $\Lambda$. We let $l_1, l_2\in \R^2$ denote lattice vectors generating the lattice:
$$\Lambda := \left\{ m_1 l_1+m_{2} l_{2} ~|~ m_i \in \Z \right\}. $$
Denote by $Y \subset \R^{2}$ a fundamental domain of the given lattice. Explicitly, we take
$$ Y := \left\{ c_1 l_1+c_{2} l_{2} ~|~ 0 \le c_1,c_{2} \le 1 \right\}. $$
The dual lattice of $\Lambda$, denoted $\Lambda^*$, is generated by $\alpha_1,\alpha_{2}$ satisfying $ \alpha_i\cdot l_j = 2\pi \delta_{ij}$. The (real) Brillouin zone $Y^*$ is now defined as $Y^*:= \R^{2} / \Lambda^*$. We remark that $Y^*$  has the topology of a two-dimensional torus. For band gap modes, the complex Brillouin zone is given by $Y^*+ \iu\R^2$.

In two dimensions, the resonators $D_1,D_2,\dots,D_N\subset Y$ are disjoint, connected sets with boundaries in $C^{1,s}$ for some $0<s<1$. As before, we let $N$  denote the number of resonators inside $Y$, and let $D = \bigcup_{i=1}^N D_i$. An example of the setting, with a single resonator in a square lattice, is sketched in \Cref{fig:square_lattice}.
 
\begin{figure}[tbh]
	\centering
	\includegraphics[]{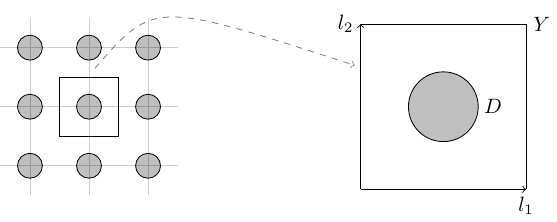}
	\caption{Illustration of the crystal in the case of a square lattice with a single resonator inside the unit cell ($N=1$).} \label{fig:square_lattice}
\end{figure}

For $\alpha \in Y^*$ in the real Brillouin zone and $\beta\in \R^2$, we consider the following two-dimensional Helmholtz problem

\begin{equation}\label{eq:u_pde}
    \begin{cases}
		\ds\Delta v + k^2u = 0, & \text{in } Y \setminus \overline{D}, \\
		\ds\Delta v + k_i^2u = 0,  & \text{in } D_i, \\
		\ds u\rvert_+ -u\rvert_- =0, &\text{on }\partial D,\\
		\ds\frac{\partial u}{\partial \nu}\bigg|_{-} - \delta \frac{\partial u}{\partial \nu} \bigg|_{+} = 0, & \text{on } \partial D, \\
        \ds u(x + \ell) = e^{\i (\alpha+\i \beta) \cdot \ell}u(x), &\text{ for all }\ell \in \Lambda,
    \end{cases}
\end{equation}
where $\nu$ stands for the outward-facing normal derivative. Here, $|_\pm$ denotes the limit from outside or inside $D$, respectively. 

\subsection{The band gap  Green's function} For a real quasimomentum $\alpha$, the $\alpha$-quasiperiodic Green's function $G^{\alpha,k}$ satisfies
\begin{equation}
    \Delta G^{\alpha,k}(x) + k^2G^{\alpha,k}(x) = \sum_{m \in \Lambda}\delta(x-m)e^{\i  \alpha \cdot m},
\end{equation}
and is given by
\begin{equation}\label{eq: real Green's function}
    G^{\alpha,k}(x) = \frac{1}{\lvert Y \rvert} \sum_{q \in \Lambda^*}\frac{e^{\i (\alpha + q)\cdot x}}{k^2-\lvert \alpha + q\rvert^2}.
\end{equation}
The solution $u(x)$ satisfies the complex Floquet condition $u(x +\ell) = e^{\i  \alpha\cdot \ell}e^{-\beta \cdot \ell}u(x)$.
Since $\beta \in \R^2$ is a vector, it defines not only the rate but also the direction, of decay. By a change of function, we may find a solution that is no longer decaying:
\begin{equation}\label{eq: exponential decay 2D}
    v(x) := e^{\beta\cdot x}u(x).
\end{equation}
Substituting into \eqref{eq:u_pde}, we find that $v$ satisfies
\begin{equation}\label{eq:v_pde}
    \begin{cases}
		\ds\Delta v - 2\beta\cdot \nabla v + (k^2+|\beta|^2)v = 0, & \text{in } Y \setminus \overline{D}, \\
		\ds\Delta v - 2\beta\cdot \nabla v + (k_i^2+|\beta|^2)v = 0, & \text{in } D_i, \\
		\ds v\rvert_+ -v\rvert_- =0, &\text{on }\partial D,\\
		\ds\frac{\partial v}{\partial \nu}\bigg|_{-} - (\beta\cdot \nu) v - \delta \left(\frac{\partial v}{\partial \nu} \bigg|_{+} - (\beta\cdot \nu) v \right) = 0, & \text{on } \partial D, \\
        \ds v(x + \ell) = e^{\i \alpha \cdot \ell}v(x), &\text{ for all }\ell \in \Lambda.
    \end{cases}
\end{equation}
We emphasise that $v$ now satisfies the real Floquet-Bloch condition given by the last line of \eqref{eq:v_pde}. The band gap  Green's function $\Tilde{G}$ can be defined in a similar fashion:
\begin{equation}
    \Tilde{G}^{\alpha,\beta,k}(x) = e^{\beta \cdot x}G^{\alpha,k}(x).
\end{equation}
As a consequence, the band gap  Green's function satisfies the following equation
\begin{equation}\label{eq: formula BAndgap greens function}
    \Delta \Tilde{G}^{\alpha,\beta,k}(x) -2 \beta\cdot \nabla \Tilde{G}^{\alpha,\beta,k}(x) + (k^2 + \lvert \beta \rvert^2) \Tilde{G}^{\alpha,\beta,k}(x) = \sum_{m \in \Lambda} e^{\i  \alpha \cdot m} \delta(x - m),
\end{equation}
and is  given by 
\begin{equation}\label{eq: Green's function}
    \Tilde{G}^{\alpha,\beta,k}(x) = \frac{1}{\lvert Y \rvert} \sum_{q \in \Lambda^*} \frac{e^{\i(\alpha + q) \cdot x}}{k^2 + \lvert \beta \rvert^2 -2\i \beta \cdot (\alpha + q) - \lvert \alpha + q\rvert^2}.
\end{equation}
The Green's function has singularities, which may arise for certain parameter values $\beta$. The series defining $\Tilde{G}^{\alpha, \beta}(x)$ breaks down if
\begin{equation}
    k^2 + \lvert \beta \rvert^2 -2\i \beta \cdot (\alpha + q) - \lvert \alpha + q\rvert^2= 0,
\end{equation}
for some $q \in \Lambda^*$. Analogously to the real Brillouin zone, we call such a singularity a \emph{Rayleigh} singularity. Rayleigh singularities occur precisely when
\begin{equation}\label{eq:Rayleigh}
    \lvert \alpha + q \rvert^2 = k^2 + \lvert \beta \rvert^2 \quad \text{and}\quad     ( \alpha  + q) \perp \beta.
\end{equation}

\begin{remark}
    This is a fundamental difference between the one-dimensional and the two- and three-dimensional lattices. In one dimension we will never have $(\alpha + q) \perp \beta$ and therefore we will never have a Rayleigh singularity if $\beta \neq 0$.
\end{remark}

\begin{remark}
    It is possible to rewrite the Green's function $(\ref{eq: Green's function})$ by taking the Fourier transform of the expression $(\ref{eq: formula BAndgap greens function})$ and then apply the Poisson summation formula. This results in
\begin{equation}\label{eq: Bandgap Green's function}
    \Tilde{G}^{\alpha,\beta,k}(x) = \sum_{m \in \Lambda}\Tilde{G}^{\alpha,\beta,k}_0(x)(x-m)e^{\i  \alpha \cdot (x-m)},
\end{equation}
where $\Tilde{G}^{\alpha,\beta,k}_0$ is the free-space Green's function given by
\begin{equation}\label{eq: Greens function kernel}
    \Tilde{G}^{\alpha,\beta,k}_0(x) = \frac{1}{(2\pi)^2}\int_{\R^2} \frac{e^{\i  \alpha \cdot x}}{k^2 + \lvert \beta \rvert^2 - 2\i \alpha\cdot \beta - \lvert \alpha \rvert^2} \mathrm{d} \alpha.
\end{equation}
\end{remark}

\subsection{Computing the complex Green's function}
For $\beta \neq 0$ the complex Green's function is given by
\begin{equation}\label{eq: lattice sum complex k}
    \Tilde{G}^{\alpha, \beta,k}(x) = \frac{1}{\lvert Y \rvert} \sum_{q \in \Lambda^*} \frac{e^{\i (\alpha + q) \cdot x}}{(k^2 + \lvert \beta \rvert^2)  - 2 \i \beta \cdot (\alpha + q) - \lvert \alpha + q\rvert^2}.
\end{equation}
For a real quasimomentum, i.e., $\beta = 0$, the Green's function coincides with the ``classical'' quasiperiodic Green's function given by $(\ref{eq: real Green's function})$. The evaluation of $\Tilde{G}^{\alpha, 0,k}(x)$ is a well-known problem, and rapid convergence rates can be achieved through Ewald summation or lattice summation methods (see, for example, \cite{toukmaji1996ewald,linton2010lattice} for comprehensive reviews). Hence, instead of computing the lattice sum $(\ref{eq: lattice sum complex k})$ directly, we compute the difference $\Tilde{G}^{\alpha, \beta,k}_{R}$:
\begin{align}
    \Tilde{G}^{\alpha, \beta,k}_{R}(x) &:= \Tilde{G}^{\alpha, \beta,k}(x) -\Tilde{G}^{\alpha, 0,k}(x)\\
    &= \frac{1}{\lvert Y \rvert}\sum_{q \in \Lambda^*}e^{\i (\alpha + q)\cdot x}\left(\frac{1}{(k^2 + \lvert \beta \rvert^2) -  2 \i \beta \cdot (\alpha + q) - \lvert \alpha + q\rvert^2}- \frac{1}{k^2 - \lvert \alpha + q \rvert^2 }\right)\\
    &=\frac{1}{\lvert Y \rvert}\sum_{q \in \Lambda^*}e^{\i (\alpha + q)\cdot x} \frac{ 2 \i \beta \cdot (\alpha + q)-\lvert \beta \rvert^2}{\bigl((k^2 + \lvert \beta \rvert^2)- 2 \i \beta \cdot (\alpha + q) -  \lvert \alpha + q \rvert^2\bigr)\left(k^2 - \lvert \alpha + q\rvert^2\right)}.\label{eq: Grenns_lattice sum}
\end{align}
In two dimensions, we observe that the series defining $\Tilde{G}^{\alpha, \beta,k}_{R}(x)$ is absolutely convergent. We may now rewrite the complex Green's function $(\ref{eq: lattice sum complex k})$ as
\begin{equation}
    \Tilde{G}^{\alpha, \beta,k}(x) = \Tilde{G}^{\alpha, 0,k}(x) + \Tilde{G}^{\alpha, \beta,k}_{R}(x).
\end{equation}
In other words, by exploiting existing summation methods for $\Tilde{G}^{\alpha, 0,k}$, we may improve the convergence rate of the lattice sum $\Tilde{G}^{\alpha, \beta,k}(x)$ by instead computing $\Tilde{G}^{\alpha, \beta,k}_{R}$.

\begin{remark}
From $(\ref{eq: Grenns_lattice sum})$ one can also directly derive an asymptotic formula, since $\Tilde{G}^{\alpha, \beta}_{R}(x) = \mathcal{O}(\beta)$ for small $\beta$. It follows that for small $\beta$,
\begin{equation}\label{eq:Gasymp}
    \Tilde{G}^{\alpha, \beta,k}(x) = \Tilde{G}^{\alpha, 0,k}(x) + \mathcal{O}(\beta),
\end{equation}
where explicit formulas for the correction terms are readily available.
\end{remark}

\subsection{Layer potential theory}
The Green's function solves the free Helmholtz equation with complex Floquet-Bloch conditions, but we still need to make sure the transmission boundary conditions on the boundary of the resonator are met. Using layer potential theory, we will rephrase the differential equation as a system of integral equations on the boundary of the resonators.
\begin{definition}[Layer potentials]
    Let $D \subset Y$ be a domain in $\R^2$ with a boundary $\partial D \in C^{1,s}$, with some $0 < s < 1$. Let $\nu$ denote the unit outward normal to $\partial D$. For $\omega > 0$, let $\tilde{\mathcal{S}}^{\alpha,\beta,k}_D$ be the single layer potential and $(\Tilde{\mathcal{K}}^{\alpha,\beta,k}_D)^*$ be the Neumann-Poincaré operator associated to the band gap Green's function $\Tilde{G}^{\alpha,\beta,k}$ defined in $(\ref{eq: Bandgap Green's function})$. For any given density $\phi \in L^2(\partial D),$
    \begin{align}
        \Tilde{\mathcal{S}}_D^{\alpha,\beta,k}[\phi](x) &= \int_{\partial D} \Tilde{G}^{\alpha,\beta,k}(x-y)\phi(y) \mathrm{d} \sigma(y), \quad x \in \R^2, \label{def: single layer potential}\\
       (\Tilde{\mathcal{K}}^{\alpha,\beta,k}_D)^*[\phi](x) &= \int_{\partial D} \frac{\partial \Tilde{G}^{\alpha,\beta,k}(x-y)}{\partial \nu(x)}\phi(y) \mathrm{d}\sigma(y), \quad x \in \partial D.
    \end{align}
\end{definition}
These operators are connected by the following so-called \emph{jump conditions}.
\begin{lemma}\label{lemma: jump conditions}
    Let $\phi \in L^2(\partial D)$ and assume that
    \begin{equation}
        \lvert \alpha + q \rvert^2 \neq \omega^2 + \lvert \beta \rvert^2 \text{ or } (\alpha + q)\cdot \beta \neq 0,
    \end{equation}
    for all $q \in \Lambda^*$. Then we have the jump conditions
    \begin{align}
        \left.\Tilde{\mathcal{S}}^{\alpha,\beta,k}_D[\phi]\right|_+ &= \left.\Tilde{\mathcal{S}}^{\alpha,\beta,k}_D[\phi]\right|_-,\\
        \left.\frac{\partial}{\partial \nu}\right|_\pm \Tilde{\mathcal{S}}^{\alpha,\beta,k}_D[\phi] &= \left( \pm \frac{1}{2}I + (\Tilde{\mathcal{K}}^{\alpha,\beta,k}_D)^*\right)[\phi].
    \end{align}
\end{lemma}
\begin{proof}
From \eqref{eq: Grenns_lattice sum}, we know that $\Tilde{G}^{\alpha, \beta,k}_{R}(x)$ is regular around $x=0$, and therefore, $\Tilde{\mathcal{S}}^{\alpha, \beta,k}_D$ satisfies the same jump conditions as $\Tilde{\mathcal{S}}^{\alpha, 0,k}_D$.
\end{proof}
We may represent the solution to the scattering problem \eqref{eq:v_pde} in terms of the single-layer potential as
\begin{equation}\label{eq:intrep}
    v(x) = \begin{cases}
        \Tilde{\mathcal{S}}_D^{\alpha,\beta,k}[\phi](x), \quad&\text{in } Y \setminus \overline{D},\\
         \Tilde{\mathcal{S}}_{D_i}^{\alpha,\beta, k_i}[\psi](x), \quad&\text{in } D_i,
    \end{cases}
\end{equation}
for some unknown densities $\phi\in L^2(\partial D)$ and $\psi_i \in L^2(\partial D_i)$. 
For brevity, we define
\begin{equation}
    \left( \mathcal{N}_D^{\alpha,\beta,k}\right)^* := \left( \Tilde{\mathcal{K}}_D^{\alpha,\beta,k}\right)^*- (\beta\cdot \nu)\Tilde{\mathcal{S}}^{\alpha,\beta,k}_D
\end{equation}
and let $\mathcal{N}_D^{\alpha,\beta,k}$ denote the $L^2$-adjoint of $\left(\mathcal{N}^{\alpha,\beta,k}_D\right)^*$. The rationale for introducing $\mathcal{N}$ is the following identity, which follows from Lemma \ref{lemma: jump conditions} and the product rule:
\begin{equation} \label{eq: SLP product rule}
    \left.\frac{\partial }{\partial \nu}\right|_\pm e^{-\beta \cdot x}\Tilde{\mathcal{S}}^{\alpha,\beta,k}_D[\phi] = e^{-\beta \cdot x}\left( \pm \frac{1}{2}I + \left( \mathcal{N}_D^{\alpha,\beta,k}\right)^*\right)[\phi].
\end{equation}
Based on \eqref{eq:intrep} and \Cref{lemma: jump conditions}, we find that \eqref{eq:v_pde} is equivalent to the system of integral equations
\begin{equation}\label{eq:intA}\A^{\alpha,\beta}(\omega,\delta)\begin{pmatrix} \psi\\ \phi\end{pmatrix} = 0\end{equation}
where the operator $\A^{\alpha,\beta}$ is given by
	\begin{equation}\A^{\alpha,\beta}(\omega,\delta) = \begin{pmatrix}
		\Tilde{\mathcal{S}}_D^{\alpha,\beta,k_i} & - \Tilde{\mathcal{S}}_D^{\alpha,\beta,k} \\ 
		-\frac{1}{2}I + \left( \mathcal{N}_D^{\alpha,\beta,k_i}\right)^*  & -\delta \left(\frac{1}{2}I + \left( \mathcal{N}_D^{\alpha,\beta,k}\right)^* \right)
	\end{pmatrix}.\end{equation}
The problem of finding the complex band structure $\omega = \omega(\alpha,\beta)$ has now been rephrased as the nonlinear eigenvalue problem \eqref{eq:intA}.
 
For general $\beta \neq 0$, the single-layer potential $\Tilde{\mathcal{S}}_D^{\alpha,\beta,k}$ might not be invertible. Nevertheless, for small $\beta$, invertibility follows from the case $\beta =0$.
\begin{lemma}
    For all $q \in \Lambda^*$, assume that 
    \begin{equation}
        \lvert \alpha + q \rvert \neq \lvert \beta \rvert \text{ or } (\alpha + q)\cdot \beta \neq 0.
    \end{equation}
    If $\alpha \neq 0$, then for small enough $\lvert \beta \rvert$, $\Tilde{\mathcal{S}}_D^{\alpha,\beta,0}: L^2(\partial D) \to H^1(\partial D)$ is invertible.
\end{lemma}
\begin{proof}
For small $\beta$, it follows from \eqref{eq:Gasymp} that 
\begin{equation}
\Tilde{\mathcal{S}}_D^{\alpha,\beta,0} = {\mathcal{S}}_D^{\alpha,0} + \O(|\beta|),
\end{equation}
with respect to the operator norm. Since $\mathcal{S}_D^{\alpha,0}$ is invertible \cite{ammari.fitzpatrick.ea2018Mathematical}, it follows that $\Tilde{\mathcal{S}}_D^{\alpha,\beta,0}$ is invertible for small enough $\beta$.
\end{proof}
As we will see, singularities of $\Tilde{\mathcal{S}}_D^{\alpha,\beta,0}$ will play an important role for the behaviour of the subwavelength gap functions.

The next result shows that, as a result of the Neumann eigenvalue $k=0$ of $-\Delta$ inside $D$, the operator $- \frac{1}{2}I + \left(\mathcal{N}^{\alpha,\beta,0}_D\right)^*$ has a nontrivial kernel of dimension $N$.
\begin{lemma}\label{lem:kernel}
    Assume that $D$ has $N$ connected components and that 
    \begin{equation}
        \lvert \alpha + q \rvert \neq \lvert \beta \rvert \text{ or } (\alpha + q)\cdot \beta \neq 0,
    \end{equation}
    for all $q \in \Lambda^*$. Then, provided that $\Tilde{\mathcal{S}}^{\alpha,\beta,0}_D$ is invertible,
    \begin{equation}
        \operatorname{ker}\left(- \frac{1}{2}I + \left(\mathcal{N}^{\alpha,\beta,0}_D\right)^*\right) \text{ and } \operatorname{ker}\left(- \frac{1}{2}I + \mathcal{N}^{\alpha,\beta,0}_D \right)    
    \end{equation}
    are $N$-dimensional and spanned, respectively, by $\psi_i$ and $\phi_i$ given by 
    \begin{equation}
        \psi_i = \left(\Tilde{\mathcal{S}}^{\alpha,\beta,0}_D\right)^{-1}[e^{\beta \cdot x} \chi_{\partial D_i}], \quad \phi_i = e^{-\beta \cdot x}\chi_{\partial D_i}, \quad i = 1, \dots, N.
    \end{equation}
\end{lemma}

\begin{proof}
    We begin with $\psi_i$. Let $u(x) = \Tilde{\mathcal{S}}^{\alpha,\beta,0}_D[\psi_i]$ and let $v(x) = e^{-\beta \cdot x}u(x)$. Then $\Delta v = 0$ in $D$ and $v(x) = \chi_{\partial D_i}$ on $\partial D$, so $v(x) = \chi_{D_i}$ in $D$. From the jump relations (Lemma \ref{lemma: jump conditions}) and the product rule $(\ref{eq: SLP product rule})$, we have $\left(-\frac{1}{2}I + \left(\mathcal{N}^{\alpha,\beta,0}_D\right)^*\right)[\psi_i] =0 $. Conversely, if $\left(-\frac{1}{2}I + \left(\mathcal{N}^{\alpha,\beta,0}_D\right)^*\right)[\psi] =0$ for some $\psi \in L^2(\partial D)$, we have that $v$ defined as $v(x) = e^{-\beta \cdot x} \Tilde{\mathcal{S}}^{\alpha,\beta,0}_D[\psi]$ satisfies the interior Neumann problem on $D$, and hence is constant.
    
    We now turn to $\phi_i$. Take $\psi \in L^2(\partial D)$ and let $v(x) = e^{-\beta \cdot x}\Tilde{\mathcal{S}}^{\alpha,\beta,0}_D[\psi]$, so that $\Delta v = 0$ in $D$. Then we have
    \begin{align}
        \left\langle \psi, \left(-\frac{1}{2}I + \mathcal{N}^{\alpha,\beta,0}_D \right)[\phi_i]\right\rangle &= \int_{\partial D_i} e^{-\beta \cdot x} \left(-\frac{1}{2}I + \left(\mathcal{N}^{\alpha,\beta,0}_D\right)^*\right)[\psi] \mathrm{d} \sigma \\
        &= \int_{\partial D_i} \frac{\partial v}{\partial \nu} \mathrm{d} \sigma\\
        &= 0,
    \end{align}
    so that $\phi_i \in \operatorname{ker}\left(- \frac{1}{2}I + \mathcal{N}^{\alpha,\beta,0}_D \right)$. Since the two kernels have equal dimensions, it follows that $\phi_i$ spans $\operatorname{ker}\left(- \frac{1}{2}I + \mathcal{N}^{\alpha,\beta,0}_D\right)$.
\end{proof}
For the limiting value $\delta = 0$, \Cref{lem:kernel} shows that $\omega=0$ is an eigenvalue of \eqref{eq:intA}. For small but nonzero $\delta$, these eigenvalues are perturbed into subwavelength eigenvalues characterised by the following capacitance theorem for band gap  frequencies.
\begin{theorem}\label{thm: subwavelength resonant frequencies}
    Consider a system of $N$ subwavelength resonators in the unit cell $Y$ and assume
    \begin{equation}
        \lvert \alpha + q \rvert^2 \neq \lvert \beta \rvert^2 \text{ or } (\alpha + q)\cdot \beta \neq 0,
    \end{equation}
    for all $q \in \Lambda^*$ and assume that $\Tilde{\mathcal{S}}^{\alpha,\beta,0}_D$ is invertible. As $\delta \to 0$, there are, up to multiplicity, $N$ subwavelength resonant frequencies $\omega^{\alpha, \beta}_n$, for $n = 1, \dots, N$, which satisfy the asymptotic formula
    \begin{equation}
        \omega^{\alpha, \beta} = \sqrt{\delta\lambda^{\alpha, \beta}_n} + \mathcal{O}(\delta), \quad n = 1,\dots, N,
    \end{equation}
    where $\{\lambda^{\alpha, \beta}_n\}$ are the $N$ eigenvalues of the generalised capacitance matrix $\mathcal{C}^{\alpha, \beta} \in \C^{N \times N}$, given by
    \begin{equation}\label{eq: def Capacitance matrix}
        \mathcal{C}_{ij}^{\alpha,\beta} = - \frac{v_i^2}{\lvert D_i \rvert}\int_{D_i} e^{-\i  \beta \cdot x}\psi_j \mathrm{d} \sigma, \quad \psi_i = (\tilde{\mathcal{S}}^{\alpha,\beta,0}_D)^{-1}[e^{\beta \cdot x} \chi_{\partial D_i}].
    \end{equation}
\end{theorem}

\begin{proof}
We follow the abstract capacitance-matrix formulation of \cite[Appendix A]{ammari2023functional}. Observe that 
\begin{equation}
    \mathcal{A}^{\alpha,\beta}(\omega, \delta) = \mathcal{A}^{\alpha,\beta}(\omega, 0) + \mathcal{L}(\omega, \delta),
\end{equation}
where
\begin{equation}
    \mathcal{L}(\omega, \delta) = - \delta \begin{pmatrix}
        0 & 0 \\
        0 & \frac{1}{2}I + (\mathcal{N}^{\alpha,\beta,k}_D)^*
    \end{pmatrix}, \quad \mathcal{L}_0 = - \begin{pmatrix}
        0 & 0 \\ 0 & \frac{1}{2}I + (\mathcal{N}^{\alpha,\beta,0}_D)^*
    \end{pmatrix}.
\end{equation}
Asymptotically, we may expand $\mathcal{L}(\omega, \delta) = \delta \mathcal{L}_0 + \mathcal{O}(\omega \delta)$, for $\omega$ and $\delta$ close to $0$. At $\delta = 0$ and for $\omega$ close to $0$, we have
\begin{equation}\label{eq:polepencil}
	\A(\omega,0)^{-1} = \frac{K}{\omega^2} + \Rc(\omega), \quad \text{ for } \quad K = \sum_{i=1}^N \langle \Phi_i, \cdot\rangle \Psi_i,
\end{equation}
where $\ker(\A(0,0)) = \mathrm{span}\{\Psi_j\}$, $\ker(\A^*(0,0)) = \mathrm{span}\{\Phi_j\}$ and $\Rc$ is holomorphic for $\omega$ in a neighbourhood of $0$.
Here, we have the kernel basis functions
\begin{equation}
    \Psi_i = \begin{pmatrix}
        \psi_i \\ \psi_i
    \end{pmatrix}, \quad \Phi_i = -\frac{v_i^2}{\lvert D_i \rvert}\begin{pmatrix}
        0 \\ \phi_i
    \end{pmatrix},
\end{equation}
for $i = 1,\dots, N$, where
\begin{equation}
    \psi_i = (\tilde{\mathcal{S}}^{\alpha,\beta, 0}_D)^{-1}[e^{\beta \cdot x} \chi_{\partial D_i}], \quad \phi_i = e^{-\beta \cdot x} \chi_{\partial D_i}, \quad i= 1, \dots, N.
\end{equation}
Then the generalised capacitance coefficients are $\mathcal{C}_{ij} = -\langle \Phi_i, \mathcal{L}_0\Psi_j \rangle.$ We then have
\begin{align}
    \mathcal{C}_{ij} &= - \frac{\delta v_i^2}{\lvert D_i \rvert}\int_{D_i} e^{-\beta \cdot x}\left(\frac{1}{2}I + (\mathcal{N}^{\alpha,\beta,0}_D)^*\right)[\psi_j]\mathrm{d} \sigma\\
    &= \frac{\delta v_i^2}{\lvert D_i \rvert}\int_{D_i} e^{-\beta \cdot x}\left(-\frac{1}{2}I + (\mathcal{N}^{\alpha,\beta,0}_D)^*\right)[\psi_j] \mathrm{d} \sigma - \frac{\delta v_i^2}{\lvert D_i \rvert}\int_{D_i} e^{-\beta \cdot x}\psi_j \mathrm{d}  \sigma\\
    &=-\frac{\delta v_i^2}{\lvert D_i \rvert}\int_{D_i} e^{-\beta \cdot x} \psi_j \mathrm{d} \sigma.\label{eq: capacitance evaluation}
\end{align}
The assertion follows by \cite[Appendix A]{ammari2023functional}.
\end{proof}

\subsection{Numerical computations}
In this section, we illustrate our results numerically using the techniques presented in Appendix \ref{Numerical Analysis} in the case where the resonators  are circles of some radius $R$. We consider both a single resonator inside the unit cell and a dimer. In the case of a dimer, we have two subwavelength bulk bands, and we explore how the band gap between them is filled by the gap functions. Throughout, we restrict to the subwavelength frequency regime and compute the complex band structure based on the capacitance formulation in \Cref{thm: subwavelength resonant frequencies}.

\subsubsection{Single resonator}
We consider a single resonator of radius $R=0.05$ in a square lattice with lattice constant $L=1$ (c.f. \Cref{fig:square_lattice}). Figures \ref{Fig: 2D_diagonal} and \ref{Fig: 2D_vertical} illustrate the band functions (black) and gap functions (red) for $\alpha$  along the irreducible Brillouin zone $\Gamma M X \Gamma$  and along two directions of $\beta$.
The plots can be understood as follows. The band functions and gap functions are functions $\omega = \omega(\alpha,\beta)$ along paths in $(\alpha,\beta)$-space such that $\omega(\alpha,\beta)$ is real. For a real quasimomentum, i.e., $\beta = 0$, this reduces to the classical band functions. For complex quasimomenta, i.e., $\beta \neq 0$, we chose beta along a specified direction and, for each $\alpha$, with magnitude given by roots of the equation
\begin{equation}\label{eq:Im}
\mathrm{Im}\bigl(\omega(\alpha,\beta)\bigr) = 0.
\end{equation}
Numerically, we solve the nonlinear equation \eqref{eq:Im} using Muller's method.

\begin{figure}[htb]
    \centering
    \subfloat[][Complex band structure: band functions (black) and gap functions (red).]{{\includegraphics[width=0.7\linewidth]{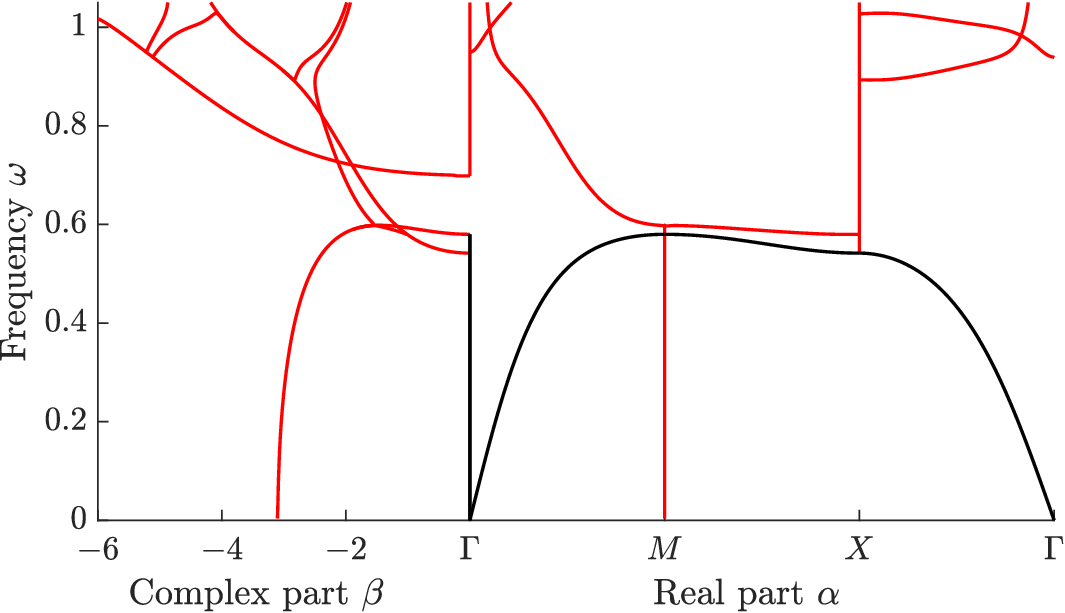}}}\hfill
    \subfloat[][Path for $\alpha$ and $\beta$.]{{\includegraphics[width=0.23\linewidth]{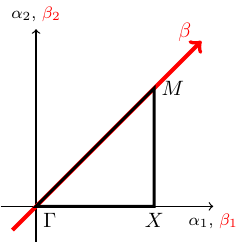} \vspace{11pt}  }}
    \caption{Complex band structure (A) for a square lattice of subwavelength resonators with diagonal evanescent direction (B). The system has a band gap above $\omega \approx 0.6$, and there are multiple branches of the complex band structure, both above and below the band edge. The left-hand side of (A) shows $\omega$ as function of $\beta$ while the right-hand side shows $\omega$ as function of $\alpha$.}
   \label{Fig: 2D_diagonal}
\end{figure}

\begin{figure}[htb]
    \centering
    \subfloat[][Complex band structure: band functions (black) and gap functions (red).]{{\includegraphics[width=0.7\linewidth]{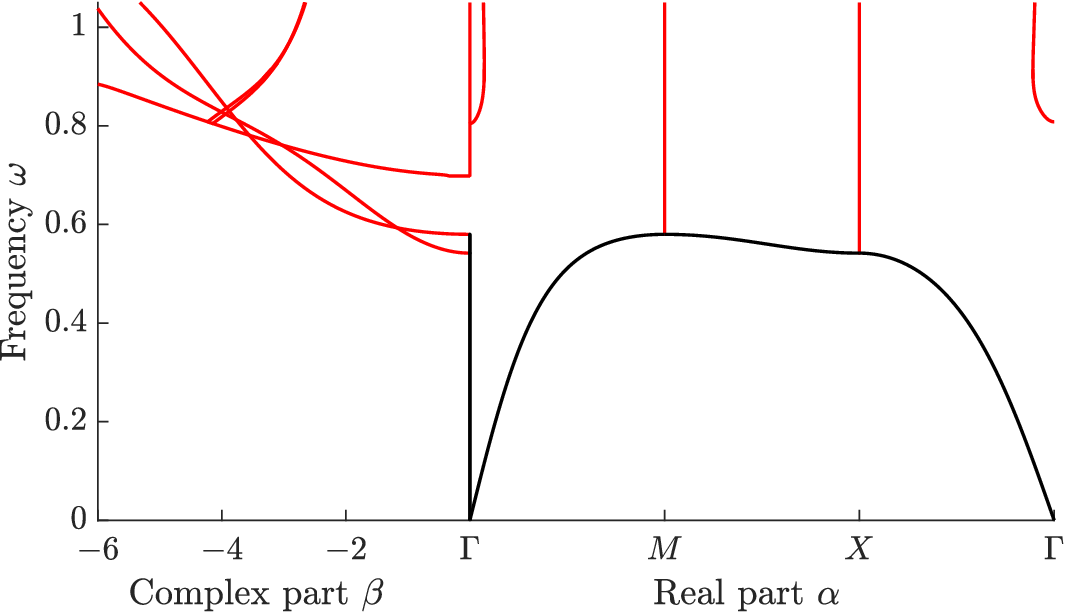} }}\hfill
    \subfloat[][Path for $\alpha$ and $\beta$.]{{\includegraphics[width=0.23\linewidth]{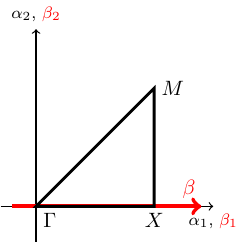} \vspace{11pt} }}
    \caption{We use the same numerical setup as in Figure \ref{Fig: 2D_diagonal}, but now take $\beta$ along a horizontal axis.}
   \label{Fig: 2D_vertical}
\end{figure}

In \Cref{Fig: 2D_diagonal} (A), we show the roots of \eqref{eq:Im} for $\beta$ along the diagonal direction (parallel to $M$). The figure format follows \cite{EvanescentWaves},  where the left-hand side shows $\omega$ as function of the complex part $\beta$, while the right-hand side shows $\omega$ as function of the real part $\beta$. The band structure has a band gap above $\omega\approx 0.6$. Nevertheless, the complex band structure covers the entire frequency range in the sense that any frequency is associated to either a bulk mode or an evanescent mode. 

\Cref{Fig: 2D_vertical} shows the same numerical setup as \Cref{Fig: 2D_diagonal}, but with horizontal evanescent direction $\beta$. Again, the complex band structure covers the entire frequency range. Interestingly, seen in both \Cref{Fig: 2D_diagonal} and  \Cref{Fig: 2D_vertical}, there are multiple branches of the gap functions, even though this structure has only a single resonator in the unit cell and a single band function. Remarkably, there are even complex branches that drop below the bulk band edge and, for $\beta$ along the diagonal, there is a branch at zero frequency $\omega =0$.

For certain values $(\alpha,\beta)$, the single-layer potential might fail to be invertible, and the capacitance matrix is singular close to these points (c.f. equation \eqref{eq: def Capacitance matrix}). By leveraging these singularities, we penetrate deeply inside the band gap while $\beta$ stays relatively small. This is of particular interest as it can be seen as a cap of the decay up to a certain band gap frequency. More concretely, as the frequency $\omega$ diverges at this point, a solution with frequency deep inside the band gap does not decay much faster than a solution with frequency close to the band edge. In the dilute regime (i.e., for small radius  $R$ of the resonators), the singularities appear periodically in $\beta$, and can be understood as perturbations of the Rayleigh singularities \eqref{eq:Rayleigh}. Figure \ref{Fig: Dilute Regime} shows the band gap function $\omega(\alpha,\beta)$ for $\alpha $ fixed at $\alpha = (\pi,\pi)$, and we observe poles of $\omega$ for $\beta$ close to the points given by \eqref{eq:Rayleigh}.
\begin{figure}[htb]
    \centering
    \subfloat[][Surface plot of the capacitance viewed from the top. Singularities under the form of peaks (yellow) and pits (dark blue).]{{\includegraphics[width=6cm]{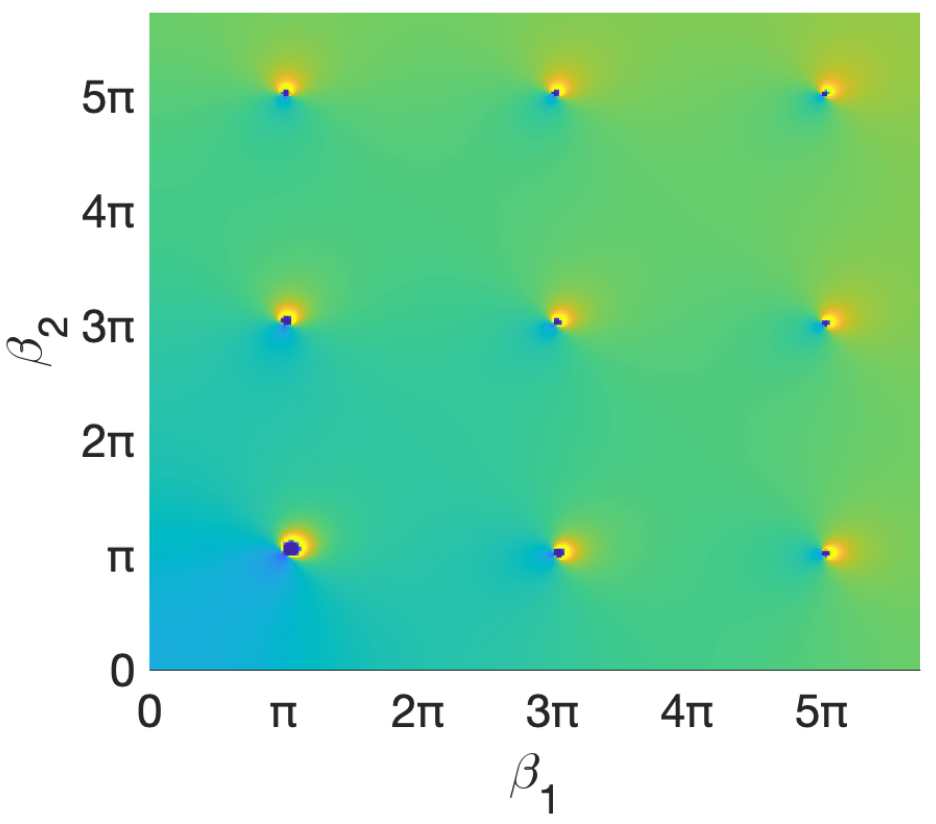} }}
    \qquad
    \subfloat[][Surface plot viewed from the side.]{{\includegraphics[width = 6cm]{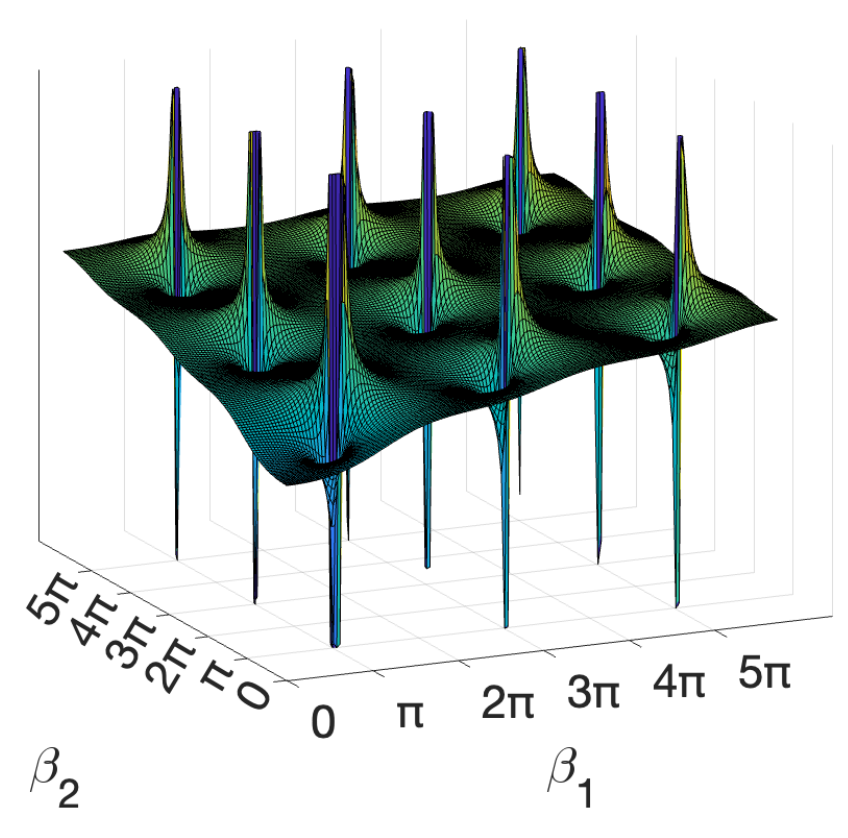} }}
    \caption{In the dilute regime, the singularities of the capacitance are located close to the points given by \eqref{eq:Rayleigh}. Computation performed for small resonators $R = 0.005$.}
   \label{Fig: Dilute Regime}
\end{figure}

\begin{remark}\label{rem: diff 2d and 1d}
There are multiple differences between the one-dimensional and two-dimensional cases. Crucially, the transfer matrix approach of \Cref{sec:general} does not generalise to two dimensions, and there are (in general) more than two real or complex bands associated to each frequency $\omega$. Moreover, there is no equivalent to Theorem \ref{Thm: alpha in the band gap} in two dimensions, and there are bands for which neither $\alpha$ nor $\beta$ is fixed to a high-symmetry point.
\end{remark}

\subsubsection{Dimer case.}
We now turn to the case of two resonators situated in the unit cell (c.f. \Cref{fig: 2D_Dimer} (B)). Again, we take $L=1$ and, for simplicity, assume that they have equal radius $R=0.05$. The  complex  band structure is shown in \Cref{fig: 2D_Dimer} (A) for diagonal $\beta$. There are now two bulk bands separated by a gap. As before, the complex band structure covers the entire frequency regime, and there are branches which fill the band gap region. 

\begin{figure}[htb]
    \centering
    \subfloat[][Complex band structure: band functions (black) and gap functions (red).]{{\includegraphics[width=0.7\linewidth]{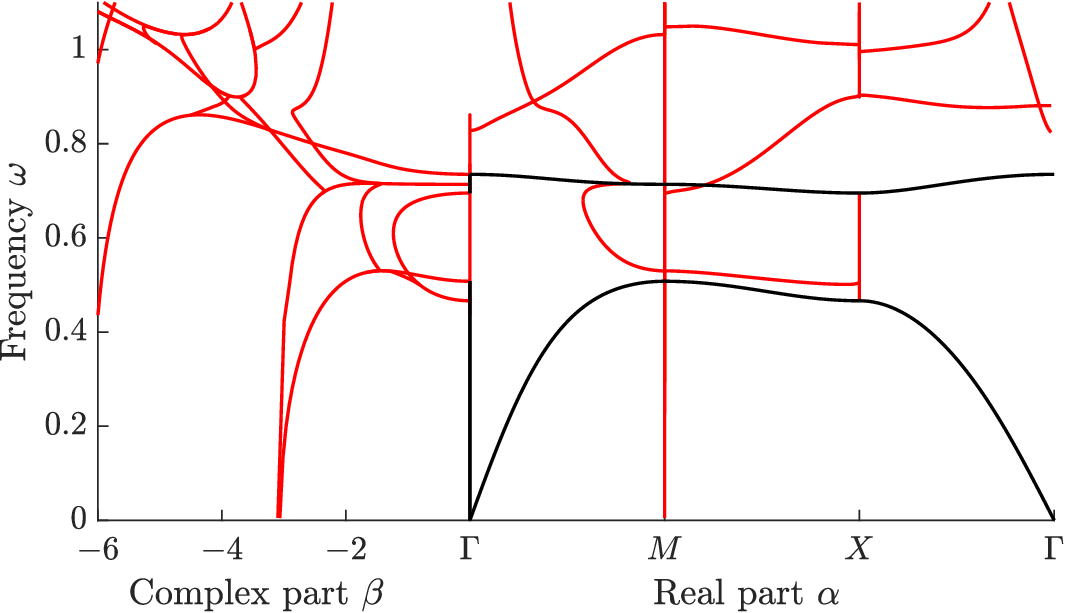}}}\hfill
    \subfloat[][Sketch of the structure.]{{\includegraphics[width=0.25\linewidth]{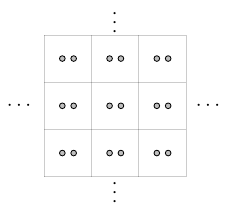} \vspace{20pt}}}
    \caption{Complex band structure (A) for a square lattice of dimers sketched in (B). In this case, there are two bulk bands separated by a gap. As before, the complex band structure covers the entire frequency range. The path for the complex quasimomentum $\beta$ is identical to Figure \ref{Fig: 2D_diagonal} (B).}
   \label{fig: 2D_Dimer}
\end{figure}

\section{Concluding remarks}\label{sec: Conclusion}
We have developed a new approach to characterise evanescent waves in band gap structures, and have derived an analytical and numerical framework on how those band functions may be generated starting from the band gap Green's function. The complex band structure covers the entire frequency range, and can be used to predict the decay rate of band gap evanescent waves. In two dimensions, we have generalised the multipole method for numerical computation of the complex band structure, and have numerically illustrated our results in a number of cases. 

The bottleneck of the numerical calculations is the evaluation of the two-dimensional series \eqref{eq:Sr}. In order to improve calculations, as well as enable complex band calculations of three-dimensional crystals, we plan to derive accelerated representations of this series. Another fundamental problem will be to prove that the decay length of topologically protected edge modes can be accurately predicted by the band gap functions of corresponding bulk structure in two-dimensional structures.

\section{Data availability}
No datasets were generated or analysed during the current study.

\appendix

\section{Numerical methods} \label{Numerical Analysis}

This section introduces numerical techniques to compute the complex band function and discusses their respective convergence rates.
\subsection{Lattice sum} Computing the single layer potential \eqref{def: single layer potential}  involves evaluating a sum over a $2$-dimensional infinite lattice. 
We will briefly touch upon the convergence of a finite, truncated lattice sum in the definition of $\Tilde{G}^{\alpha, \beta}_{R}(x)$. By using the integral test, one can easily show that the lattice sum is absolutely convergent. 
\begin{figure}[ht]
    \centering
    \includegraphics[width = 12cm]{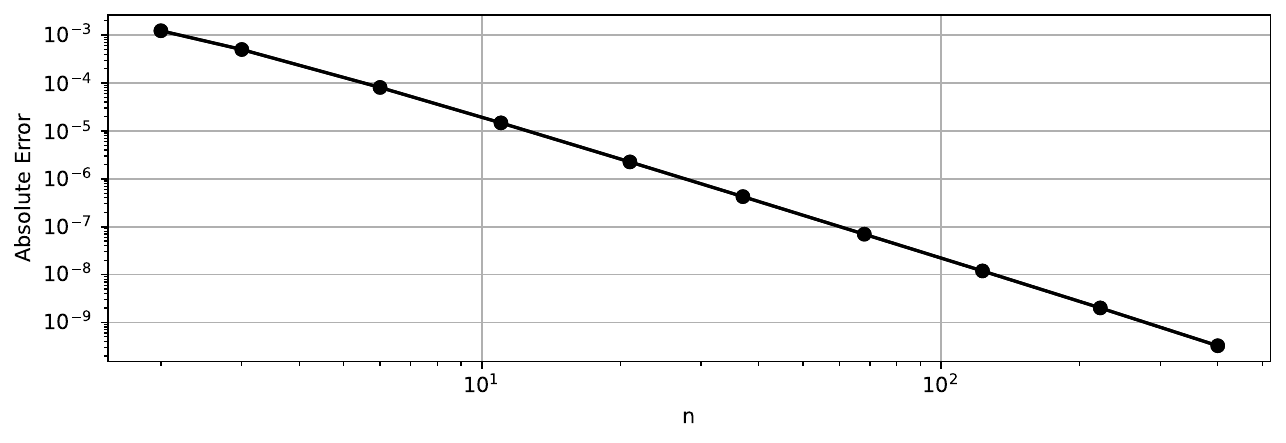}
    \caption{Error in the truncated lattice sum for a lattice of size $[-n, n] \times [-n, n]$ against a lattice with size $n = 1000$. Computed across varying parameters $\alpha$, $\beta$, and $\omega$.}
    \label{fig: convergence rate}
\end{figure}
The numerical example for the truncation error performed in Figure \ref{fig: convergence rate} reveals that the convergence rate is approximately of order 3. For numerical computations, we will therefore consider a lattice of size $n = 10$.
Because of the slow convergence of the lattice sum of $\Tilde{G}^{\alpha, \beta}$ we have to revert to a different technique, see \cite[Section 5.4.2.3.]{ammari.fitzpatrick.ea2018Mathematical} for more details on the efficient computation of the lattice sum in the setting of subwavelength resonators.

\subsection{Multipole expansion method}
Since both lattice sums defining $\Tilde{G}^{\alpha, 0}$ and $\Tilde{G}^{\alpha, \beta}_R$ are convergent, we can define the single layer potential associated to the complex Green's function
\begin{align}
     \Tilde{\mathcal{S}}_D^{\alpha,\beta, \omega}[\phi](x) &= \int_{\partial D} \Tilde{G}^{\alpha, \beta, \omega}(x-y)\phi(y) \mathrm{d}\sigma(y)\\
     &= \underbrace{\int_{\partial D} \Tilde{G}^{\alpha, 0, \omega}(x-y)\phi(y) \mathrm{d} \sigma(y)}_{ = \Tilde{\mathcal{S}}_D^{\alpha, \omega}} + \underbrace{ \int_{\partial D} \Tilde{G}^{\alpha, \beta, \omega}_{R}(x-y)\phi(y) \mathrm{d} \sigma(y)}_{ =\Tilde{\mathcal{S}}^{\alpha,\beta, \omega}_R}.
\end{align}
In the case of a circular resonator $D$ we may implement the multipole expansion method (see, for example, \cite[Appendix C]{BandGapBubbly}) to represent the single layer potential in a Fourier basis. This is of particular interest as for numerical computations we can use a truncated Fourier basis to achieve a finite matrix representation of the single layer potential. The density function $\varphi = \varphi(\theta)$ is a $2\pi$-periodic function, so it admits the following Fourier series expansion
\begin{equation}
    \varphi = \sum_{n \in \Z} a_n e^{\i  n \theta}.
\end{equation}
We now compute the remainder term of the single layer potential $\Tilde{\mathcal{S}}^{\alpha,\beta, \omega}_R[e^{\i n\theta}]$. As the single layer potential consists of an integral over the boundary of the circular resonator, we may express the variables $x = r e^{\i  \theta}$, $y = r e^{\i  \varphi}$ and $\mathrm{d}\varphi(y) = r \mathrm{d}\varphi$ in polar coordinates, where $r$ designates the radius of the resonators. The resulting representation of the single layer potential is
\begin{equation}\label{eq: single layer potential compuation}
    \Tilde{\mathcal{S}}^{\alpha,\beta, \omega}_R[e^{\i n\varphi}](re^{\i \theta}) = \int_0^{2\pi} \Tilde{G}^{\alpha, \beta}_{R}\bigl(r(e^{\i  \theta}- e^{\i  \varphi})\bigr) e^{\i n\varphi} r \mathrm{d}\varphi.
\end{equation}
Evaluating the following scalar product yields the matrix of the operator in the  Fourier basis,
\begin{align}
    (\Tilde{\mathcal{S}}^{\alpha,\beta, \omega}_R)_{m, n} &= \frac{1}{2\pi}\langle e^{\i m \theta},\Tilde{\mathcal{S}}^{\alpha,\beta, \omega}_R[e^{\i n\varphi}] \rangle\\ &= \frac{1}{2\pi}\int_0^{2\pi} e^{-\i m\theta}\Tilde{\mathcal{S}}^{\alpha,\beta, \omega}_R[e^{\i n\varphi}](re^{\i \theta}) \mathrm{d} \theta.
\end{align}
Truncating to $\lvert m \rvert,\lvert n \rvert  \leq K$ makes $(\Tilde{\mathcal{S}}^{\alpha,\beta, \omega}_R)_{m, n}$ a finite matrix. In other words, the single layer potential is represented by a finite (square) $(2K+1)$-dimensional matrix.
The same procedure can be applied to $\Tilde{\mathcal{S}}_D^{\alpha, \omega}$. This yields a matrix representation for the single layer potential for a complex quasimomentum
\begin{equation}\label{eq: discrete single layer potential}
    (\Tilde{\mathcal{S}}_D^{\alpha,\beta, \omega})_{m, n} = ( \Tilde{\mathcal{S}}^{\alpha, \omega}_D)_{m, n} +  (\Tilde{\mathcal{S}}^{\alpha,\beta, \omega}_R)_{m, n} \in \R^{(2K+1)\times(2N+1)}.
\end{equation}
The convergence rate of the multipole method is exponential. For numerical computations, we choose $K = 5$.

\subsection{Evaluating the single layer potential}
We now outline the numerical evaluation of the single layer potential for the remainder term. We seek to avoid using any kind of numerical integrators by evaluating the integrals explicitly. Since the lattice sum is absolutely convergent by the dominated convergence theorem we may interchange integration and summation over the lattice.
The idea is to integrate each term in the lattice sum first and then take the lattice sum over the previously computed integral. As a consequence, 
{\small \begin{align}
    &(\Tilde{\mathcal{S}}_R^{\alpha,\beta, \omega})_{m, n} = \frac{1}{2\pi}\int_0^{2\pi} e^{-\i m\theta} \int_0^{2\pi} \Tilde{G}^{\alpha, \beta,\omega}_{R}\bigl(r(e^{\i  \theta}- e^{\i  \varphi})\bigr) e^{\i n\varphi} r \mathrm{d}\varphi \mathrm{d}\theta \label{eq: SLP integral definition}\\
    \qquad &={\frac{r}{2 \pi \lvert Y \rvert}\int_0^{2\pi}  e^{-\i m\theta} \int_0^{2\pi}\sum_{q \in \Lambda^*} \frac{ e^{\i (\alpha + q)\cdot (r(e^{\i  \theta}- e^{\i  \varphi}))} \left(2 \i \beta \cdot (\alpha + q) - \lvert \beta \rvert^2 \right)}{\bigl(\lvert \alpha + q \rvert^2 + 2 \i\beta \cdot (\alpha + q) - (\omega^2 + \lvert \beta \rvert^2) \bigr)\left(\lvert \alpha + q\rvert^2 - \omega^2 \right)} e^{\i n\varphi}  \mathrm{d}\varphi \mathrm{d}\theta}\\
    &= \frac{r}{2 \pi \lvert Y \rvert}\sum_{q \in \Lambda^*} \frac{ \left(2 \i \beta \cdot (\alpha + q) - \lvert \beta \rvert^2\right) \int_0^{2\pi}  e^{-\i m\theta} \int_0^{2\pi}e^{\i (\alpha + q)\cdot (r(e^{\i  \theta}- e^{\i  \varphi}))}  e^{\i n\varphi}  \mathrm{d}\varphi \mathrm{d}\theta}{\bigl(\lvert \alpha + q \rvert^2 + 2 \i \beta \cdot (\alpha + q) - (\omega^2 + \lvert \beta \rvert^2) \bigr)\left(\lvert \alpha + q\rvert^2 - \omega^2 \right)}.\label{eq: SLP sum then integral}
\end{align}}\\%
We evaluate the integral inside the sum explicitly
\begin{equation}\label{eq: term-wise integral}
    \int_{0}^{2\pi}\int_0^{2\pi} e^{\i (\alpha + q)\cdot (x-y)}e^{-\i m \theta}e^{\i n\varphi}\mathrm{d} \varphi \mathrm{d}\theta.
\end{equation}
Here, $x = re^{\i \theta}$ while $y =r e^{\i \varphi}$. The generating function for the Bessel function is
\begin{equation}
    e^{\frac{1}{2}z(t-t^{-1})} = \sum_{m = - \infty}^\infty t^m \mathbf{J}_m(z).
\end{equation}
When we set $t = e^{\i  \theta}$, we obtain
\begin{equation}
    e^{\frac{1}{2}z(e^{\i \theta}-e^{-\i \theta})} = \sum_{m = - \infty}^\infty e^{\i m \theta }\mathbf{J}_m(z),
\end{equation}
so that 
\begin{equation}
    \int_0^{2\pi} e^{\i  z \sin(\theta)}e^{-\i m\theta} \mathrm{d} \theta = 2 \pi \mathbf{J}_m(x).
\end{equation}
Let us rotate the vector $(\alpha + q)$ such that it becomes parallel to the $x$-axis. Let $\psi = \operatorname{arg}(\alpha + q)$. In other words, we do the following change of variables:
\begin{align}
    \theta' &= \theta - \psi,\\
    \varphi' &= \varphi - \psi.
\end{align}
This way, $(\alpha + q)\cdot(x -y) = r \lvert(\alpha + k)\rvert(\cos(\theta')- \cos(\varphi'))$.
This allows us to rewrite \eqref{eq: term-wise integral}  as 
\begin{align}
    &\int_0^{2\pi} \int_0^{2\pi} e^{\i (\alpha + q) \cdot (x-y)} e^{-\i m \theta}e^{\i n \varphi} \mathrm{d} \varphi \mathrm{d}\theta\\
   & = \int_0^{2\pi} e^{\i r\lvert (\alpha + q) \rvert\cos(\theta')}e^{-\i m\theta'}e^{-\i m\psi}\mathrm{d}\theta' \int_0^{2\pi} e^{\i r\lvert(\alpha + q)\rvert \cos(\varphi')}e^{\i n \varphi'}e^{\i n \psi} \mathrm{d}\varphi'\\
   &= e^{\i m\left(\frac{\pi}{2}- \psi\right)}(2\pi \mathbf{J}_m(r\lvert (\alpha + q)\rvert)) e^{-\i n\left(\frac{\pi}{2}- \psi\right)}(2\pi \mathbf{J}_{-n}(r\lvert(\alpha + q )\rvert))\\
   &= (2\pi)^2 e^{\i  \psi(n-m)}\i^{m-n} (-1)^n\mathbf{J}_{m}(r\lvert(\alpha + q )\rvert)\mathbf{J}_{n}(r\lvert(\alpha + q )\rvert).
\end{align}
As a consequence, \eqref{eq: SLP sum then integral}  can be rewritten as,
{\small
\begin{equation}\label{eq:Sr}
    (\Tilde{\mathcal{S}}_R^{\alpha,\beta, \omega})_{m, n} = \frac{2\pi r}{\lvert Y \rvert}\sum_{q \in \Lambda^*} \frac{ \left(2 \i \beta \cdot (\alpha + q) - \lvert \beta \rvert^2\right)  e^{\i  \psi(n-m)}\i^{m-n} (-1)^n\mathbf{J}_{m}(r\lvert(\alpha + q )\rvert)\mathbf{J}_{n}(r\lvert(\alpha + q )\rvert)}{\bigl(\lvert \alpha + q \rvert^2 + 2 \i \beta \cdot (\alpha + q) - (\omega^2 + \lvert \beta \rvert^2) \bigr)\left(\lvert \alpha + q\rvert^2 - \omega^2 \right)}.
\end{equation}}\\%
As the integral in \eqref{eq: SLP integral definition}  has been computed in closed form, we do not introduce any numerical quadrature error. As a consequence, the numerical error introduced in the evaluation of the entries of the single layer potential solely stems from the truncation of the lattice sum and the multipole expansion method.
\subsection{Computation of the Fourier coefficients} When evaluating the capacitance matrix \eqref{eq: capacitance evaluation}  we also need to represent $e^{\beta \cdot x}$ in the Fourier basis. The Fourier coefficients are given by
\begin{align}
    f_n &= \frac{1}{2\pi}\int_{0}^ {2\pi} e^{\beta \cdot x} e^{-\i  n \theta} \mathrm{d} \theta \label{eq: exp(beta x) Fourier}\\
    &= \frac{1}{2\pi} \int_{0}^ {2\pi} e^{r \lvert \beta \rvert \cos(\theta)}e^{-\i n \theta}e^{-\i  n \psi} \mathrm{d}\theta,
\end{align}
where  $\psi = \operatorname{arg}(\beta)$. We do a change of variables $\theta' = \theta - \psi$, resulting in
\begin{align}
    f_n &= \frac{1}{2\pi} e^{\i n \left(\frac{\pi}{2}- \psi \right)}\int_0^{2\pi} e^{r \lvert \beta \rvert \sin(\theta)}e^{-\i  n \theta} \mathrm{d} \theta \\
    &= \frac{1}{2\pi} e^{\i n\left(\frac{\pi}{2}- \psi \right)}2\pi\mathbf{J}_n(-r \lvert \beta \rvert \i) \\
    &= (-1)^n e^{\i n \left(\frac{\pi}{2}- \psi \right)} \i^n \mathbf{I}_n(r \lvert \beta \rvert)\\
    &= e^{-\i n \psi}\mathbf{I}_n(r \lvert \beta \rvert),
\end{align}
where $\mathbf{I}_n(r\lvert \beta \rvert)$ is the modified Bessel function.

\bibliographystyle{abbrv}
\bibliography{bibliography_master}

\end{document}